\documentclass[12pt]{article}
\usepackage{amsfonts}
\usepackage{graphicx}
\usepackage{latexsym,amsmath, color}
\usepackage{natbib}

\newcommand{\cvp}{\mbox{$\stackrel{p}{\longrightarrow}$}}

\def\1{\mathbb{I}}

\newcounter{thm}[section]
\newcounter{appen}

\newtheorem{theor}[thm]{Theorem}

\newtheorem{lem}[thm]{Lemma}

\newtheorem{assumption}[appen]{Assumption}
\newtheorem{lemma}[appen]{Lemma}

\newenvironment{proof}[1][Proof]{\textbf{#1.} }{\ \rule{0.5em}{0.5em}}

\begin{document}

\title{Powerful nonparametric checks for  quantile regression}
\author{Samuel Maistre\thanks{
CREST (Ensai), France. Email: samuel.maistre@ensai.fr }, \quad Pascal Lavergne\thanks{
Toulouse School of Economics, France.  Email: pascal.lavergne@univ-tlse1.fr } \quad and
\quad Valentin Patilea\thanks{
CREST (Ensai), France. Email: patilea@ensai.fr }}
\date{}
\maketitle

\begin{abstract}
We address the issue of lack-of-fit testing for a parametric quantile
regression. We propose a simple test that involves one-dimensional
kernel smoothing, so that the rate at which it detects local
alternatives is independent of the number of covariates.  The test has
asymptotically gaussian critical values, and wild bootstrap can be
applied to obtain more accurate ones in small samples.  Our procedure
appears to be competitive with existing ones in simulations.  We
illustrate the usefulness of our test on birthweight data.


\end{abstract}

\newpage

\section{Introduction}

\setcounter{equation}{0}

Quantile regression, as introduced by \cite{Koenker1978}, has
emerged as an alternative to mean regression.  It allows for a richer
data analysis by exploring the effect of covariates at different
quantiles of the conditional distribution of the variable of interest.
Parametric quantile regression generalizes usual regression are is
particularly valuable if variables have asymmetric distributions or
heavy tails.  Koenker's monograph (2005) and the review of \cite{Yu2003}
detail the theory and practice of quantile regression.

As in any statistical modeling exercice, it is crucial to check the
fit of a parametric quantile model.  There has been a large effort
devoted to testing of the fit of parametric mean regressions, however
only few lack-of-fit tests of parametric quantile regressions. \cite{He2003} 
extend the approach of \cite{Stute1997} and is based on a
vector-weighted cumulative summed process of the residuals.  \cite{Bierens2002} 
generalize the integrated conditional moment test of \cite{Bierens1997} 
to quantile regression.  In both cases, the limit
distribution of the test statistic is a non-linear functional of a
Gaussian process, so that implementation may require rather involved
computations to obtain critical values.  \cite{Zheng1998} use kernel
smoothing over the design space, to obtain an asymptotically pivotal
test statistic.  \cite{Horowitz2002} extend such an approach
and propose an adaptive procedure to choose the smoothing parameter.
As in any multidimensional nonparametric problem, the curse of
dimensionality may be detrimental to the performances of the test,
see e.g. \cite{Lavergne2012} for illustrations.

In this paper, we introduce a new testing methodology that avoids
multidimensional smoothing, but still yield an omnibus test.  Our test
has three specific features. First, it does not require smoothing with
respect to all covariates under test.  This allows to mitigate the
curse of dimensionality that appears with nonparametric smoothing,
hence improving the power properties of the test.  Second, the test
statistic is asymptotically pivotal, while wild bootstrap can be used
to obtain small samples critical values of the test. This yields a
test whose level is well controlled by bootstrapping, as shown in
simulations.  Third, our test equally applies whether some of the
covariates are discrete.

The paper is organized as follows. In Section 2, we present our
testing procedure, we study its asymptotic behavior under the null
hypothesis and under a sequence of local alternatives, and we
establish the validity of wild bootstrap.  In Section 3, we compare
the small sample behavior of our test to some existing procedures, and
we illustrate its use on birthweight data. Section 3 concludes.  Section
4 gathers our technical proofs.

\section{Lack-of-Fit Test for Quantile Regression}

\subsection{Principle and Test}

Consider modeling the quantile of a real random variable $Y$
conditional upon covariates $Z \in\mathbb{R}^q$, $q\geq 1$.  We assume
that $Z=(W,X^\prime)^\prime$, where $W$ is continuous and admits a
density with respect to the Lebesgue measure, while $X$ may include
both continuous and discrete variables.  Formally, if $F(\cdot \mid
z)$ denotes the conditional distribution of $Y$ given $Z=z$, the
$\tau$-th conditional quantile is
$
Q_{\tau}(z)=\inf \{y:F(y\mid z)\geq \tau \}
$.
Assuming $F(\cdot \mid z)$ is absolutely continuous for almost all
$z$, this is equivalent to $F(Q_{\tau }(z)\mid z)=\tau $.
The parametric quantile regression model of interest posits that the
conditional $\tau$-th quantile of $Y$ is given by $g (Z;\beta_{0})$,
where $g (\cdot; \beta)$ is known up to the parameter vector $\beta
\in B\subset \mathbb{R}^{p}$, that is,
\begin{equation}
Y=g(Z;\beta_{0})+ \varepsilon,  \qquad F \left( g (Z;\beta_{0}) \mid Z\right) =  \tau
\, .
\label{qu1}
\end{equation}
 The validity of our parametric quantile regression is thus equivalent to
\begin{equation}
{H}_{0}:\ \exists \, \beta_{0} \in B \ : \ F(g(Z;\beta _{0})\mid Z)-\tau =\mathbb{E}\left\{ \mathbb{I}\{Y\leq
g(Z;\beta _{0})\}-\tau \mid Z\right\} =0 \, \text{a.s.}
\label{qu2}
\end{equation}
Hence testing the the correct specification of our parametric quantile
regression models reduces to testing a zero conditional mean hypothesis.
The alternative hypothesis is then
\begin{equation*}
{H}_{1}:\ \mathbb{P}\left[ \mathbb{E}\left\{ \mathbb{I}\{Y\leq g(Z;\beta
)\}-\tau \mid Z\right\} =0\right] <1\quad \text{for any }\beta \in B
\, .
\end{equation*}

The key element of our testing approach is the following lemma. See also \cite{Lavergne2014} for a related result. First
let us introduce some notation. Hereafter, if
$g:\mathbb{R}^k\rightarrow \mathbb{R}$ is an integrable function,
$\mathcal{F}[g]$ denotes its Fourier transform, that is
$$
\mathcal{F}[g] (t)=  \int_{\mathbb{R}^k} \exp(- 2\pi i t^\prime u ) g(u) du \, .
$$
\begin{lem}
\label{Fundamental-Lemma}
Let $\left(W_{1},\, X_{1},\, U_{1}\right)$
and $\left(W_{2},\, X_{2},\, U_{2}\right)$ be two independent draws of
$\left(W,\, X,\, u\right)$,  and
$K(\cdot)$ and $\psi(\cdot)$  even functions with (almost
everywhere) positive Fourier integrable transforms. Define
\[
I\left(h\right)=\mathbb{E}\left[U_{1}U_{2} h^{-p}K\left(\left(W_{1}-W_{2}\right)/h\right)
\psi\left(X_{1}-X_{2}\right)\right]
\, .
\]
Then for any $h>0$,
$
\mathbb{E}\left[U\mid W,X\right]= 0\,\; a.s.
\Leftrightarrow
I (h) = 0
$.
\end{lem}
{\bf Proof.}
Let $\langle \cdot, \cdot \rangle$ denote the standard inner product and
 $\mathcal{F}\left[K\right]$ be the Fourier transform of $K(\cdot)$.
Using Fourier Inversion Theorem, change of variables, and elementary
properties of conditional expectation,
\begin{eqnarray*}
I(h) & = &
\mathbb{E}\left[U_{1}U_{2}\int_{\mathbb{R}^{p}}e^{2\pi
    i\langle t, \; W_{1}-W_{2}\rangle
  }\mathcal{F}\left[K\right]\left(th\right)dt
  \int_{\mathbb{R}^{q}}e^{2\pi i \langle s
    ,\;X_{1}-X_{2}\rangle}\mathcal{F}\left[\psi\right]\left(s\right)ds\right]\\ &
= &
\int_{\mathbb{R}^{q}}\int_{\mathbb{R}^{p}}\left|\mathbb{E}\left[\mathbb{E}\left[U\mid
    W,X\right]e^{2\pi i\left\{ \langle t, W
    \rangle + \langle s, X\rangle \right\}
  }\right]\right|^{2}\mathcal{F}\left[K\right]\left(th\right)\mathcal{F}\left[\psi\right]\left(s\right)dtds
\, .
\end{eqnarray*}
Since the Fourier transforms $\mathcal{F}\left[K\right]$ and
$\mathcal{F}\left[\psi\right]$  are  strictly positive, $I(h)=0$
iff
\[
\mathbb{E}\left[\mathbb{E}\left[U\mid W,X\right]
e^{2\pi i\left\{ \langle t, W\rangle + \langle s, X\rangle \right\}}\right] = 0
\quad \forall t, s
\Leftrightarrow
\mathbb{E}\left[U\mid W,X\right] = 0 \qquad \mbox{a.s.}
\qquad \rule{0.5em}{0.5em}
\]


From the above results, it is sufficient to test whether $I (h)=0$ for
any arbitrary $h$.  We chose to consider a sequence of $h$ decreasing
to zero when the sample size increases, which is one of the ingredient
that allows to obtain a tractable asymptotic distribution for the test
statistic.  Assume we have at hand a random sample $(Y_i,W_i,X_i)$,
$1\leq i\leq n$, from $(Y,W,X)$.  Then we can estimate
$I\left(h\right)$ by the second-order U-statistic
\begin{equation*}
I_{n}\left( \beta_{0} \right) = I_{n}\left( \beta_{0};h \right) =
\frac{1}{n(n-1)}\sum\limits_{1\leq j\neq i\leq n}U_{i}\left( \beta_{0}
\right) U_{j}\left( \beta_{0} \right) \frac{1}{h} K_{h}\left( W_{i}-W_{j}
\right) \psi( X_{i} - X_{j})
\end{equation*}
where $U_{i}(\beta )=\mathbb{I}\{Y_{i}\leq g(Z_{i};\beta )\}-\tau$ and
$K_{h}(\cdot )=K(\cdot /h)$.

For estimating $\beta_{0}$, we follow \cite{Koenker1978}, who
showed that under (\ref{qu1}) a consistent estimator of
$\beta _{0}$ is  obtained by
minimizing
\begin{equation}
\arg\min_{\beta}
\sum_{i=1}^{n}\rho _{\tau}\left( Y_{i}-g(Z_{i};\beta )\right)
\, ,
\label{est1}
\end{equation}
where $\rho _{\tau }(e)=\left( \tau -\mathbb{I}(e<0) \right) \, e$ is
the so-called check function.  While this is not a differentiable
optimization problem, it is convex and tractable, see e.g.  \cite{Koenker2005} 
for some computational algorithms.  Let us define
\begin{equation}
T_{n}=nh^{1/2}\frac{I_{n}(\widehat{\beta })}{v_{n}}
\quad \mbox{where }\quad
v_{n}^{2} =\frac{2\,\tau ^{2}(1-\tau )^{2}}{n(n-1)}
\sum\limits_{j\neq i}h^{-1}K_{h}^2\left( W_{i}-W_{j} \right) \psi^2( X_{i} - X_{j})
\label{var_est1}
\;.
\end{equation}
An asymptotic $\alpha$-level test of  ${H}_{0}$ is then
\begin{quote}
Reject $H_{0}$ if $T_{n}\geq z_{\alpha }$, where $z_{\alpha}$
is the $(1-\alpha )-$quantile of the standard normal distribution.
\end{quote}
Our test statistic is very similar to the one proposed by \cite{Zheng1998}, 
but the latter uses smoothing on all components of $Z$ while
we smooth only on the first component $W$.

The statistic $v_{n}^{2}$ is the variance of $nh^{1/2}I_{n}(\beta
_{0})$ conditional on the $Z_{i}$ under ${H}_{0}$.  In general,
$v_{n}^{2}$ does not consistently estimate the conditional variance of
$nh^{1/2}I_{n}(\beta )$ under the alternative hypothesis. In some
cases $v_{n}^{2}$ overestimates this conditional variance (this is
certainly the case for misspecified median regression model because
$\tau (1-\tau )$ attains the maximum value at $\tau =1/2$), so that
the test may suffer some power loss. In a mean regression context,
\cite{Horowitz2001} and \cite{Guerre2005} proposed
to use a nonparametric estimator of the conditional variance.  This
might be adapted to quantile regression, but in simulations our test
appears to be well-behaved and more powerful than competitors, so we
decided in favor of the simplest estimator $v_{n}^{2}$.


\subsection{Behavior Under the Null Hypothesis}
\label{sec_null}

To derive the asymptotic properties of our lack-of-fit test, we
introduce our set of assumptions on the data-generating process, the
parametric model (\ref{qu1}), the functions $K(\cdot)$ and
$\psi(\cdot)$, and the bandwidth $h$.

\begin{assumption}
\label{qas1} (a) The random vectors $(\varepsilon_{1},Z_{1}^{\prime })^{\prime
},\ldots ,(\varepsilon_{n},Z_{n}^{\prime })^{\prime }$ are independent copies of the
random vector $(\varepsilon,Z^{\prime })^{\prime }\in \mathbb{R}^{1+q}$. The conditional
$\tau$th quantile of $\varepsilon$ given $Z=\left(W,X'\right)'$ is equal to zero.

(b) The variable $W$ admits an absolutely continuous density with the
respect of the Lebesgue measure on the real line.

(c) The conditional density $f_{\varepsilon}(\cdot \mid z)$ of
$\varepsilon$ given $Z=z$ is uniformly bounded.  There exists $a>0$
such that $f_{\varepsilon}(\cdot \mid z)$ is differentiable on
$(-a,a)$ for any $z$ with $\left\vert f_{\varepsilon}^{\,\prime
}(0\mid z)\right\vert \leq C\infty$.  Moreover, the derivatives
$f_{\varepsilon}^{\,\prime }(\cdot \mid z)$ satisfy a uniform
H\"{o}lder continuity condition, that is there exist positive
constants $C_2$ and $c$ independent of $z$ such that
$\forall\left\vert u_{1}\right\vert ,\left\vert u_{2}\right\vert \leq
a$, $\left\vert f_{\varepsilon}^{\,\prime }(u_{1}\mid
z)-f_{\varepsilon}^{\,\prime }(u_{2}\mid z)\right\vert \leq
C_2\left\vert u_{1}-u_{2}\right\vert ^{c}$.
\end{assumption}

\begin{assumption}
\label{qas2} (a) The parameter space $B$ is a compact convex subset of $\mathbb{R}^{p}$.
$\beta _{0}$ is the unique solution of
$
\min_{B} \mathbb{E}
\left[ \rho _{\tau} \left( Y -g(Z,\beta ) \right)\right]
$ and is an interior point of $B$.

(b) The matrix
\[
\mathbb{E}\left[ f_{\varepsilon}(0\mid Z\,)
\frac{\partial}{\partial \beta}{g}(Z;\beta _{0})\frac{\partial}{\partial \beta'}{g}^{\prime}(Z;\beta _{0})\right]
\]
is finite and nonsingular.

(c) There exists functions $A\left( \cdot
\right) $, $B\left( \cdot \right) $, and $D\left( \cdot \right) $,
with $\mathbb{E}[A^{4}(\,Z\,)]$, $\mathbb{E} [B^{2}(\,Z\,)]<\infty$,
and $\mathbb{E}[D^{4}(\,Z\,)]$,  such that
\begin{equation*}
\left\Vert \frac{\partial}{\partial \beta}{g}(z;\beta )\right\Vert \leq A\left( z\right) ,\qquad
\left\Vert \frac{\partial}{\partial \beta}{g}(z;\beta )
\frac{\partial}{\partial \beta'}{g}^{\prime }(z;\beta )\right\Vert \leq
D(z)\qquad \text{for any }\beta
\, ,
\end{equation*}
\begin{equation*}
\left\Vert \frac{\partial}{\partial \beta}{g}(z;\beta _{1})-
\frac{\partial}{\partial \beta}{g}(z;\beta _{2})\right\Vert \leq
B(z)\left\Vert \beta _{1}-\beta _{2}\right\Vert \qquad \text{for any }
z,\beta _{1},\beta _{2}
\, .
\end{equation*}

(d) The class of functions $\{ g(Z;\beta) : \beta \in B \}$ is a
Vapnik-\v{C}ervonenkis (VC) class.
\end{assumption}
\begin{assumption}
\label{K}
(a) The function $K(\cdot )$ is a bounded symmetric univariate density
of bounded variation with positive Fourier transform.

(b) The function $\psi(\cdot )$ is a bounded symmetric multivariate function
with positive Fourier transform.

(c) $h\rightarrow 0$ and $n^\alpha h^2\rightarrow \infty $ for some
$\alpha\in(0,1)$ as $n\rightarrow \infty$.
\end{assumption}
Our assumptions combine standard assumptions for parametric quantile
regression estimation and specific ones for our lack-of-fit test.
Among the latter, the conditions on the error term $\varepsilon$
impose neither independence of $\varepsilon$ and $Z$, nor a specific
form of dependence such as $\varepsilon= s\left(Z\right) e$ with $e$
independent of $Z$ as in \cite{He2003}.  Assumption \ref{qas2}(d)
is a mild technical condition that guarantees suitable uniform rates
of convergence for some $U-$processes appearing in the proofs. This
condition is satisfied for many parametric models, for instance when
$g(Z,\beta) = q(Z^\prime \beta)$ with
$q:\mathbb{R}\rightarrow\mathbb{R}$ monotone or of bounded variation,
see e.g.  van der Vaart and Wellner (1996, Section 2.6).  Also, if
there is $\beta\in B$ such that $g(Z,\beta)$ is squared integrable,
then Assumption \ref{qas2}(d) follows from \ref{qas2}(c).
 Assumptions on $K(\cdot)$ allows for
the use of a triangular, normal, logistic, Student (including Cauchy),
or Laplace densities.  For $\psi(\cdot)$, one can choose e.g.
$\psi(x)=\exp(-\|x\|^2)$, or any multivariate extension of the aforementioned densities.
Restrictions on the bandwidth are compatible with optimal choices for
regression estimation, see e.g. \cite{Hardle1985}, and for
regression checks, see \cite{Guerre2002} and \cite{Horowitz2002}. 
The following theorem states the asymptotic validity
of our test.
\begin{theor}
\label{as_law} Under the Assumptions \ref{qas1} to \ref{K}, the test based on
$T_{n}$ has asymptotic level $\alpha$ under $H_{0}$.
\end{theor}

\subsection{Behavior under Local Alternatives}
\label{sec_loc_alt}

We now investigate the behavior of our test when $H_{0}$ does not
hold, and specifically we consider a sequence of local alternatives of
the form
\begin{equation}
{H}_{1n}:\ Y=g(Z;\beta_{0})+ r_{n}\delta (Z) + \varepsilon, \qquad F
\left( g (Z;\beta_{0}) \mid Z\right) = \tau
\, ,
\label{loc_alt}
\end{equation}%
where $r_n$, $n\geq 1$, is a sequence of real numbers tending to zero
and $\delta(Z)$ is a real-valued function satisfying
\begin{equation}
\mathbb{E}\left[ f_{\varepsilon}(0\mid Z\,)\delta (Z)
\frac{\partial}{\partial \beta}{g}(Z;\beta _{0})\right]
=0\quad \mbox{\rm and }\quad 0<\mathbb{E}[\delta ^{4}(Z)]<\infty
\, .
\label{conddelta}
\end{equation}
This condition ensures that our sequence of models (\ref{loc_alt})
does not belong to the null hypothesis $H_{0}$.  We do not impose any
smoothness restriction on the function $\delta (\cdot ) $ as is
frequent in this kind of analysis, see e.g. \cite{Zheng1998}.  As shown in
Lemma \ref{conv_alt} in the Proofs section, $\widehat{\beta }-\beta
_{0}=O_{\mathbb{P}}(n^{-1/2}+r_{n}^{2})$ under ${H}_{1n}$.  Our next
result states that these local alternatives can be detected whenever
$r_{n}^{2}n h^{1/2}\rightarrow \infty$. Hence our test \emph{does not
  suffer from the curse of dimensionality} against local alternatives,
since its power is unaffected by the number of regressors.
\begin{theor}
\label{l_power} Under Assumptions \ref{qas1} to \ref{K},  the test based on $T_{n}$ is
consistent against the sequence of alternatives ${H}_{1n}$
with $\delta (Z)$ satisfying (\ref{conddelta}) if $r_{n}^{2}nh^{1/2}
\rightarrow \infty$.
\end{theor}

\subsection{Bootstrap Critical Values}
\label{sec_bootstrap}

The asymptotic approximation of the behavior of $T_{n}$ may not be
satisfactory in small samples as is customary in smoothing-based
lack-of-fit tests. This motivates the use of bootstrapping for
obtaining critical values.  The distribution of $T_{n}$ depends weakly
on the distribution of the error term $\varepsilon$, because
$\mathbb{I}\{Y\leq g(Z;\beta _{0})\}-\tau $ under $H_{0}$ is a
Bernouilli random variable irrespective of the particular distribution
of $\varepsilon$. The same phenomenon is noted by \cite{Horowitz2002} 
for their test statistic.  Their proposal is thus to
naively (or nonparametrically) bootstrap from the empirical
distribution of the  residuals.  This is a valid bootstrap
procedure when errors are identically distributed, and it remains
asymptotically valid for non identically distributed errors.  A first
possibility is thus to adopt naive residual bootstrap for our test.
Alternatively, \cite{He2003} note that one could use any
continuous distribution with the $\tau$-th quantile equal to 0.  This
constitutes a second possibility.  While asymptotically valid, these
two methods do not account for potential heteroscedastic errors.  Thus
a third possibility is the wild bootstrap method for quantile
regression introduced by \cite{Feng2011}.  The wild bootstrap
procedure for our test works as follows.
\begin{enumerate}
\item Let $\widehat \varepsilon_i = Y_i - g(Z_i; \widehat \beta)$,
  $1\leq i\leq n$, and $w_1,\cdots w_n$ be bootstrap weights generated
  independently from a two-point mass distribution with probabilities
  $1-\tau$ and $\tau$ at $2(1-\tau)$ and $-2\tau$.
Compute  $\varepsilon_{i}^{\ast }= w_i |\widehat \varepsilon_i|$ and
$Y_{i}^{\ast }=g(Z_{i};\widehat{\beta})+\varepsilon_{i}^{\ast }$
for each $i=1,...,n$.
\item Use the bootstrap data set $\left\{ Y_{i}^{\ast},Z_{i} :
  i=1,...,n\right\} $ to compute the estimator $
  \widehat{\beta}^{\ast}$, the new $U_{i}^\ast
  (\widehat{\beta}^{\ast})= \mathbb{I}\{Y_{i}^{\ast }\leq
  g(Z_{i};\widehat{\beta }^{\ast })\}-\tau $, and the new test
  statistic $T_{n}^{\ast }$.
\item Repeat Steps 1 et 2 many times, and estimate the
  $\alpha$-level critical value $z^{\ast}_{\alpha}$ by the $\left( 1-\alpha \right)$-th
  quantile of the empirical distribution of $T_{n}^{\ast}$.
\end{enumerate}
The bootstrap test  then rejects $H_{0}$ if $T_{n}\geq z^{\ast}_{\alpha }$.
Alternatively, one could resample residuals in Step 1 by naive
bootstrap, or obtain $\varepsilon_{i}^{\ast }$ by random draws from
e.g. a uniform law on the interval $[-\tau ,1-\tau ]$.  The following
theorem yields  the asymptotic validity of the bootstrap test.
\begin{theor}
\label{as_law_boot}
Under the conditions of Theorem \ref{as_law},
\begin{equation*}
\sup_{t\in \mathbb{R}}\left\vert \mathbb{P}\left( T_{n}^{\ast }\leq t\mid
Y_{1},Z_{1},...,Y_{n},Z_{n}\right) -\Phi (t)\right\vert \cvp 0
\, ,
\end{equation*}
where $\Phi \left( \cdot \right) $ is the standard normal distribution function.
\end{theor}

\section{Numerical Evidence}
\label{simul_sec}

\subsection{Small Sample Performances}

We investigated the performances of our procedure for testing
lack-of-fit of a linear median regression for two setups considered by
\cite{He2003}, namely
\begin{eqnarray}
Y & = & 1+W+X+\delta \, \left(W^{2}+WX+X^{2}\right)+\varepsilon
\, ,
\label{eq:setup1}\\
Y & = & \delta \, \log\left(1+W^{2}+X^{2}\right)+\varepsilon
\, ,
\label{eq:setup2}
\end{eqnarray}
where $W$ follows a standard normal, and $X$ independently follows a
binomial of size $5$ and probability of success $0.5$.  For the error
term, we considered the three distributions
$\mathcal{N}\left(0,\,1\right)$, $\log\mathcal{N}\left(0,\,1\right)-1$
and $\mathcal{N}\left(0,\,\left(1+W^{2}\right)/2\right)$.

For implementation, we chose $\psi(\cdot)$ as the standard normal density
and $K(\cdot)$ as triangle density with variance one.
We set $\delta=0$ in Model (\ref{eq:setup1}) to evaluate the
comparative performances of the three possible bootstrapping
procedures.  Figure \ref{fig:Level} reports ou results based on $5000$
replications for a sample size of $n=100$ at nominal level $10\%$,
when the bandwidth is $h= c n^{-1/5}$ with $c$ varying.
The three bootstrap methods yield accurate
levels for any bandwidth choice when errors are identically
distributed, while the use of asymptotic critical values yield large
underrejection.  In the heteroscedastic case, however, only the wild
bootstrap yield an empirical level close to $10\%$, while the use of
naive or uniform bootstrap results in a severely oversized test.

Next, we investigated the power of our test for Models
(\ref{eq:setup1}) and (\ref{eq:setup2}) with either standard gaussian
or heteroscedastic gaussian errors. We compared our test to the one
proposed by He and Zhu (2003, hereafter HZ), based on
\[
\max_{\left\Vert \mathbf{a}\right\Vert = 1}
n^{-1}\sum_{i=1}^{n}\left(\mathbf{a}^{\prime}\mathbf{R}_{n}\left(X_{i}\right)\right)^{2}
\quad \mbox{where } \quad \mathbf{R}_{n}\left(\mathbf{t}\right)=
n^{-1/2}\sum_{j=1}^{n} \left( \tau - \mathbb{I}\left[ U_{j} \left(
  \widehat{\beta} \right) <0\right] \right) Z_{j}
\mathbb{I}\left(Z_{j}\leq\mathbf{t}\right)
\, .
\]
We also computed the statistic proposed by \cite{Zheng1998}, which in our
setup writes
\[
\dfrac{h^{q/2}}{ \widetilde{\sigma} (n-1)} \sum_{j\neq i}
U_{i}\left(\widehat{\beta}\right)U_{j}\left(\widehat{\beta}\right)
h^{-q} \tilde{K}\left(\dfrac{W_{i}-W_{j}}{h},\,\dfrac{X_{i}-X_{j}}{h}\right)
\]
where $\widetilde{\sigma}^{2}= \frac{2
  \tau^{2}\left(1-\tau\right)^{2}}{n(n-1)} \sum_{j\neq i} h^{-q}
\tilde{K}^{2}\left(\dfrac{W_{i}-W_{j}}{h},\,\dfrac{X_{i}-X_{j}}{h}\right)$,
and $\tilde{K}$ is a triangle kernel applied to the norm of its argument.  We
apply the wild bootstrap procedure to compute the critical values of
all tests.  Figure \ref{fig:Power} reports power curves of the
different tests as a function of $\delta$ based on $2500$
replications. For the linear Model (\ref{eq:setup1}), all tests
perform almost similarly.  Our test is a bit more powerful, especially
for a larger bandwidth, which was expected given our theoretical
analysis.  For the nonlinear Model (\ref{eq:setup2}), the power
advantage of our test is more pronounced.  Its power can be as large
as twice the power of the test by \cite{He2003}.

\subsection{Empirical Illustration}

We studied some parametric quantile models for children birthweight
using data analyzed by \cite{Abrevaya2001} and \cite{Koenker2001},
who gave a detailed data description.  We focused on median
regression and the 10th percentile quantile regression.  Models are
estimated and tested on a subsample of 1168 smoking college graduate
mothers.  We first analyzed the simple model considered by \cite{He2003}, 
which is linear in weight gain during pregnancy (WTGAIN),
average number of cigarettes per day (CIGAR), and age (AGE). When
implementing our test, we chose age as the $W$ variable, and we
standardize all explanatory variables. Other details are identical to
what was done in simulations.  For both quantiles, HZ test does not
reject this specification.  Our test does not reject the linear median
regression at 10\% level, but detects misspecification for the lower
decile regression when $c=2$.

Since the more detailed analysis of \cite{Abrevaya2001} and \cite{Koenker2001} 
suggests that birthweight is quadratic in age, we then
considered this variation. None of the tests detects a misspecified
model.  Finally, we considered a more complete model similar to \cite{Abrevaya2001}, 
where we added the explanatory binary variables BOY (1 if
child is male), BLACK (1 if mother is black), MARRIED (1 if married),
and NOVISIT (1 if no prenatal visit during the pregnancy).  HZ test does
not reject the model at either quantiles.  Our test however indicates
a misspecified median regression model at 10\% level, while it does
not reject the model for the lower decile. 
Our limited empirical exercice suggests that our new test, beside existing 
procedures such as the test by \cite{He2003},  is a valuable 
addition to the practitioner toolbox.

\small
\bibliographystyle{ecta}
\bibliography{ManuscritAbre}
\normalsize

\normalsize
\newpage

\section{Proofs}
\setcounter{equation}{0} \setcounter{appen}{0}

We first recall some definitions. For the definition of a VC-class, we refer to Section 2.6.2  of \cite{Vaart1996}. Next, let $\mathcal{G}$ be a class of
real-valued functions on a set $\mathcal{S}$. We call $\mathcal{G}$ an \emph{%
Euclidean(c,d)} family of functions, or simply \emph{Euclidean}, for the envelope $G$ if there
exists positive constants $c$ and $d$ with the following properties: if $%
0<\epsilon \leq 1$ and $\lambda $ is a measure for which $\int G^{2}d\lambda
<\infty $, then there are functions $g_{1},\dots ,g_{j}$ in $\mathcal{G}$
such that (i) $j\leq c\epsilon ^{-d}$; and (ii) for each $g$ in $\mathcal{G}$
there is an $g_{i}$ with $\int |g-g_{i}|^{2}d\lambda \leq \epsilon ^{2}\int
G^{2}d\lambda $. The constants $c$ and $d$ must not depend on $\lambda $.
See e.g. \cite{Nolan1987} or \cite{Sherman1994}. Recall that if $\mathcal{F}$
is a VC-class of functions then the class $\{\mathbb{I}\{f\geq 0\}:f\in
\mathcal{F}\}$ is Euclidean for the envelope $F\equiv 1$, see \cite{Vaart1996} 
Lemma 2.6.18(iii) and Theorem 2.6.7 or  \cite{Pakes1989}.
Bellow, we shall use this property with the VC-classes of functions of
$\{\varepsilon + g(Z,\beta_0) - g(Z,\beta) : \beta \in B\}$
and $\{\varepsilon + g(Z,\beta_0) + r_n\delta(Z) - g(Z,\beta) : \beta \in B\}$.

In the following, $F_{\varepsilon}\left( \cdot \mid x\right) $ is the conditional
distribution function of $\varepsilon$ given $Z=z;$ that means  $F_{\varepsilon}\left( 0 \mid \cdot \right) \equiv \tau $.
Below $C$, $C_{1}$, $C_{2}$,... denote constants, not
necessarily the same as before and possibly changing from line to line.

\subsection{Proof of Theorem \protect\ref{as_law}}


\begin{proof}
First, we prove that if $\mathcal{H}_{0}$ holds%
\begin{equation}
n\sqrt{h}\left\{ W_{n}(\widehat{\beta })-W_{n}(\beta _{0})\right\}
=o_{\mathbb{P}}\left( 1\right) .  \label{hrte}
\end{equation}%
Let us introduce some simplifying notation:
\begin{equation}
G_{i}\left( \beta ,\beta _{0}\right) =g(Z_{i};\beta
)-g(Z_{i};\beta _{0}),\,\,\,\psi_{ij}=\psi(X_{i}-X_{j}),\,\,\,K_{h,ij}=
K_{h}\left( W_{i}-W_{j}\right) .
\label{simp_not}
\end{equation}%
Under $\mathcal{H}_{0}$%
\begin{eqnarray*}
W_{n}(\beta ) &=&\frac{h^{-1}}{n(n-1)}\sum\limits_{j\neq i}\left[ \mathbb{I}%
\{Y_{i}\leq g(Z_{i};\beta )\}-\tau \right] \left[ \mathbb{I}\{Y_{j}\leq
g(Z_{j};\beta )\}-\tau \right] K_{h,ij} \psi_{ij}\\
&=&\frac{h^{-1}}{n(n-1)}\sum\limits_{j\neq i}\left[ \mathbb{I}\{\varepsilon_{i}\leq
G_{i}(\beta ,\beta _{0})\}-F_{\varepsilon}\left( 0\mid Z_{i}\right) \right]  \\
&&\qquad \qquad \qquad \qquad \qquad \times \left[ \mathbb{I}\{\varepsilon_{j}\leq
G_{j}(\beta ,\beta _{0})\}-F_{\varepsilon}\left( 0\mid Z_{j}\right) \right] K_{h,ij}\psi_{ij}.
\end{eqnarray*}%
By a Taylor expansion, decompose
\begin{equation*}
F_{\varepsilon}\left( 0\mid Z_{i}\right) =F_{\varepsilon}\left( G_{i}(\beta ,\beta _{0})\mid
Z_{i}\right) -f_{\varepsilon}\left( 0\mid Z_{i}\right) \dot{g}^{\prime }(Z_{i};\beta
_{0})\left( \beta -\beta _{0}\right) +O_{\mathbb{P}}\left( \left\Vert \beta -\beta
_{0}\right\Vert ^{2}\right) .
\end{equation*}%
We can write $W_{n}(\beta ) - W_{n}(\beta_0 ) = \{W_{1n}^{0}(\beta ) -  W_{1n}^{0}(\beta_0 )\}  + 2W_{2n}^{0}(\beta
)+W_{3n}^{0}(\beta )+R_{n}^{0}$ where
\begin{multline*}
W_{1n}^{0}(\beta )=\frac{h^{-1}}{n(n-1)}\!\sum\limits_{j\neq i}\left[
\mathbb{I}\{\varepsilon_{i}\leq G_{i}(\beta ,\beta _{0})\}-F_{\varepsilon}\left( G_{i}(\beta
,\beta _{0})\mid Z_{i}\right) \right]  \\
\times \left[ \mathbb{I}\{\varepsilon_{j}\leq G_{j}(\beta ,\beta _{0})\}-F_{\varepsilon}\left(
G_{j}(\beta ,\beta _{0})\mid Z_{j}\right) \right] K_{h,ij}\psi_{ij}
\end{multline*}%
$W_{2n}^{0}(\beta )=\left( \beta -\beta _{0}\right) ^{\prime }\widetilde{W}%
_{2n}^{0}(\beta )$ with
\begin{eqnarray*}
\widetilde{W}_{2n}^{0}(\beta ) &=&\frac{h^{-1}}{n(n-1)}\sum\limits_{j\neq i}%
\left[ \mathbb{I}\{\varepsilon_{i}\leq G_{i}(\beta ,\beta _{0})\}-F_{\varepsilon}\left(
G_{i}(\beta ,\beta _{0})\mid Z_{i}\right) \right]  \\
&&\qquad \qquad \qquad \qquad \times f_{\varepsilon}\left( 0\mid Z_{j}\right) \dot{g}%
(Z_{j};\beta _{0})K_{h,ij}\psi_{ij},
\end{eqnarray*}%
$W_{3n}^{0}(\beta )=\left( \beta -\beta _{0}\right) ^{\prime }\widetilde{W}%
_{3n}^{0}\left( \beta -\beta _{0}\right) $ with
\begin{equation*}
\widetilde{W}_{3n}^{0}=\frac{h^{-1}}{n(n-1)}\sum\limits_{j\neq i}f_{\varepsilon}\left(
0\mid Z_{i}\right) \dot{g}(Z_{i};\beta _{0})\dot{g}^{\prime }(Z_{j};\beta
_{0})f_{\varepsilon}\left( 0\mid Z_{j}\right) K_{h,ij}\psi_{ij}=O_{\mathbb{P}}(1).\
\end{equation*}%
The rate of $\widetilde{W}_{3n}^{0}$ follows simply by computing its mean and variance.
By Assumption \ref{qas1}(c) and Assumption \ref{qas2}(c)  it is easy to check that $\left\vert R_{n}^{0}\right\vert \leq
\left\Vert \beta -\beta _{0}\right\Vert ^{2}O_{\mathbb{P}}\left( 1\right) .$ For
deriving the order of $\widetilde{W}_{2n}^{0},$ apply
Hoeffding decomposition   and write $h\widetilde{W}_{2n}^{0}(\beta
)=V_{n}^{2}(\beta )+V_{n}^{1}(\beta )$ with $V_{n}^{1}$, $V_{n}^{2}$
degenerate $U-$processes or order 1 and 2, respectively. In view of
Assumptions \ref{qas2}(d) and \ref{K}(a), apply Corollary 4 of \cite{Sherman1994} and deduce that $%
V_{n}^{2}(\beta )=O_{\mathbb{P}}\left( n^{-1}\right) $ uniformly in $\beta $ (and $h$).
Next,
if $\dot{g}^{\left( l\right) }$ denotes the $l$th component of the vector of
first-order derivatives $\dot{g},$ $1\leq l\leq p,$ and
\begin{equation*}
\pi ^{\left( l\right) }\left( Z_{i} \right) =\mathbb{E}\left[  f_{\varepsilon}\left( 0\mid Z_{j}\right) \dot{g}^{\left( l\right)
}(Z_{j};\beta _{0})
h^{-3/4}K_{h,ij} \psi_{ij} \mid Z_{i}\right]
\end{equation*}%
we can rewrite the $l$th component of the vector $V_{n}^{1}(\beta )$ as%
\begin{equation*}
\frac{h^{3/4}}{n}\sum\limits_{i=1}^{n}\left[ \mathbb{I}\{\varepsilon_{i}\leq
G_{i}(\beta ,\beta _{0})\}-F_{\varepsilon}\left( G_{i}(\beta ,\beta _{0})\mid
Z_{i}\right) \right]\pi ^{\left( l\right) }\left(
Z_{i}\right) .
\end{equation*}%
By H\"{o}lder inequality, Assumption \ref{qas1}(c), Assumption \ref{qas2}(c)
and a change of variables,
\begin{eqnarray*}
\left\vert \pi ^{\left( l\right) }\left( X_{i} \right) \right\vert
&\leq& \mathbb{E}\left[  f_{\varepsilon}\left( 0\mid
Z_{j}\right) \left\vert\dot{g}^{\left( l\right) }(Z_{j};\beta _{0}) \right\vert h^{-3/4}K_{h,ij} \left\vert\psi_{ij} \right\vert\mid Z_{i} \right]   \\
&\leq & C_1 \mathbb{E}^{1/4}\left[ A^{4}( Z_{j} )\right] \mathbb{E}^{3/4}\left[
h^{-1}K_{h,ij}^{4/3}\mid Z_{i} \right]
\\
& \leq & C_2,
\end{eqnarray*}
for any $1\leq l\leq p.$
Now, by Corollary 4 of \cite{Sherman1994}, $h^{-3/4}V_{n}^{1}(\beta
)=O_{\mathbb{P}}\left( n^{-1/2}\right) $ uniformly in $\beta .$ Deduce that%
\begin{equation*}
\sup_{\beta }|W_{2n}^{0}(\beta )|\leq \left\Vert \beta -\beta
_{0}\right\Vert O_{\mathbb{P}}\left( h^{-1}n^{-1}+h^{-1/4}n^{-1/2}\right) .
\end{equation*}%
Finally, by Lemma 1 of \cite{Zheng1998}, for any $\alpha \in (0,1)$%
\begin{equation*}
\sup_{\beta }|W_{1n}^{0}(\beta )-W_{1n}^{0}(\beta _{0})|=O_{\mathbb{P}}\left(
h^{-1}n^{-1-\alpha /4}\right)
\end{equation*}%
uniformly over $O_{\mathbb{P}}\left( n^{-1/2}\right) $ neighborhoods of $\beta _{0}.$
Gathering the results and using Lemma \ref{conv_alt} with $\delta(\cdot)\equiv 0$ we obtain (\ref{hrte}). Now, it remains to check that $%
nh^{1/2}W_{n}(\beta _{0})/v_{n}$ converges in law to a standard normal
distribution. This result easily follows as a particular case of Lemma \ref{hall_heyde} below.
\end{proof}

\subsection{Proof of Theorem \protect\ref{l_power}}

First, we derive the behavior of $\widehat{\beta },$ the estimator of $\beta
_{0}$ under the sequence of local alternatives $H_{\!1n}.$
\begin{lem}
\label{conv_alt} Suppose that  Assumptions \ref{qas1}, \ref{qas2} hold,
 let $\delta(\cdot)$ be a function such that Condition
(\ref{conddelta}) holds, and let $r_n$, $n\geq 1$ be a sequence of real
  numbers such that $r_n\rightarrow 0$. If $\widehat{\beta }=\arg
\min_{\beta \in B}\!\Gamma _{n}\left( \beta \right) $ with $\Gamma
_{n}\left( \beta \right) =\!\!\sum_{i=1}^{n}\rho _{\tau
}(Y_{i}-\!g(Z_{i};\beta ))$, then under
\negthinspace $\mathcal{H}_{0},$ $\widehat{\beta }-\!\beta _{0}\!=\!O_{\mathbb{P}}(n^{-1/2})$
and under $H_{\!1n}$ defined in (\ref{loc_alt}), $\widehat{\beta }-\beta
_{n}=O_{\mathbb{P}}(n^{-1/2})$ where
\begin{equation*}
\beta _{n}=\beta _{0}-r_{n}^{2}\left[ \mathbb{E}\left[ f_{\varepsilon}(0\mid Z\,)\dot{g%
}(Z;\beta _{0})\dot{g}^{\prime }(Z;\beta _{0})\right] \right] ^{-1}\mathbb{E}%
\left[ f_{\varepsilon}^{\,\prime }(0\mid Z)\delta ^{2}(Z)\dot{g}(Z;\beta _{0})%
\right] .
\end{equation*}
\end{lem}

\begin{proof} It is easy to check that
\begin{equation}
\left\vert \rho _{\tau }\left( a-b\right) -\rho _{\tau }\left( a\right)
\right\vert \leq \left\vert b\right\vert \max \left( \tau ,1-\tau \right)
\leq \left\vert b\right\vert .  \label{h_ineq}
\end{equation}%
Combine this with the Mean Value Theorem and Assumption \ref{qas2}(c) to
check the conditions of Lemma 2.13 of \cite{Pakes1989} and to derive
the Euclidean property for an integrable envelope for the family of
functions $\left\{ \left( y,z\right) \!\mapsto \!\rho _{\tau }(y-\!g(z;\beta
))\!:\beta \in B\right\} .$

Next, we study the consistency  of $\widehat \beta$ under $H_0$. By the uniform
law of large numbers,  $\sup_{\beta }\left\vert
n^{-1}\Gamma _{n}\left( \beta \right) -\mathbb{E}\left[ \rho _{\tau
}(Y-g(Z ;\beta ))\right] \right\vert \rightarrow 0,$ in probability
(use for instance Lemma 2.8 of Pakes and Pollard 1989).\ This uniform
convergence, the identification condition in Assumption \ref{qas2}(a), the
continuity of $g\left( z;\cdot \right) $ for any $z,$ and usual arguments
used for proving consistency of argmax estimators, allow to deduce $\widehat{%
\beta }-\beta _{0}=o_{\mathbb{P}}(1).$ To obtain the consistency
under the local alternatives approaching $H_0$, it suffices to prove $\sup_{\beta \in
B}\left\vert \Delta _{n}\left( \beta \right) \right\vert \rightarrow 0$ in
probability, where
\begin{equation*}
\Delta _{n}\left( \beta \right) =\frac{1}{n}\sum_{i=1}^{n}\left\{ \rho
_{\tau }\left( l(\varepsilon_{i},Z_{i};\beta )+r_{n}\delta (Z_{i})\right) -\rho _{\tau
}\left( l(\varepsilon_{i},Z_{i};\beta )\right) \right\}
\end{equation*}%
and $l(u,z;\beta )=u+g(z;\beta _{0})-g(z;\beta ).$ By inequality
(\ref{h_ineq}),
\begin{equation*}
\left\vert \Delta _{n}\left( \beta \right) \right\vert \leq \frac{\left\vert
r_{n}\right\vert }{n}\sum_{i=1}^{n}\left\vert \delta (Z_{i})\right\vert .
\end{equation*}%
Consequently, $ \Delta _{n}\left( \beta \right) = o_{\mathbb{P}} (1)$ uniformly over $\beta \in B$, and thus the consistency follows.

Define $\psi _{\tau }(e)=\tau -\mathbb{I}(e<0)$ as the derivative of $\rho
_{\tau }.$ To obtain the rate of convergence of $\widehat{\beta }$ under $\mathcal{H}_{1n}$
(in particular under $H_0$ by taking $r_n \equiv 0$) consider the empirical process
\begin{multline*}
\nu _{n}\left( \beta \right) = \frac{1}{\sqrt{n}}\sum_{i=1}^{n}\left\{ \psi
_{\tau }\left( Y_{i} - g(Z_{i};\beta ) \right) - \mathbb{E} [ \psi
_{\tau }\left( Y_{i} - g(Z_{i};\beta ) \right) \mid Z_{i} ] \right\} \dot{g}(Z_{i};\beta ) \\
=\frac{1}{\sqrt{n}}\sum_{i=1}^{n}\left\{ \psi
_{\tau }\left( l(\varepsilon_{i},Z_{i};\beta )+r_{n}\delta (Z_{i})\right)
-\mathbb{E}\left[ \psi _{\tau }\left( l(\varepsilon_{i},Z_{i};\beta
)+r_{n}\delta (Z_{i})\right) \mid Z_{i}\right] \right\} \dot{g}(Z_{i};\beta )
\end{multline*}%
indexed by $\beta .$ First, let us notice that
\begin{equation}
\nu _{n}\left( \beta \right) -\nu _{n}\left( \beta _{0}\right) =o_{\mathbb{P}}\left(
1\right)   \label{negli1}
\end{equation}%
uniformly over $o_{\mathbb{P}}\left( 1\right) $ neighborhoods of $\beta _{0},$ as a
consequence of Corollary 8 of \cite{Sherman1994}. Indeed, by Lemma 2.13 of \cite{Pakes1989}, 
the class of functions $\{\dot{g}(\cdot ;\beta ):\beta
\in B\}$ is Euclidean for a squared integrable envelope. Next, by the
VC-class property of the regression functions $\{g(\cdot ;\beta )$, $\beta
\in B\}$, the class of functions $\{(u,z)\mapsto \psi _{\tau }\left(
l(u,z;\beta )+r_{n}\delta (z)\right) :\beta \in B\}$ is \emph{Euclidean(c,d)}
for a constant envelope. See Lemma 2.12 of \cite{Pakes1989}.
Moreover, the constants $c$ and $d$ can be taken independent of $n$, see,
for instance, the proof of Lemma 2.6.18(v) of \cite{Vaart1996}. 
Finally, by repeated applications of the Mean Value Theorem and
Assumptions \ref{qas1}(c) and \ref{qas2}(c), for any $z,\beta _{1},\beta _{2}
$ we have
\begin{eqnarray}
&&\hskip-1cm\left\vert \,\mathbb{E}\left[ \psi _{\tau }\left( l(\varepsilon,z;\beta
_{1})+r_{n}\delta (z)\right) \right] -\mathbb{E}\left[ \psi _{\tau }\left(
l(\varepsilon,z;\beta _{2})+r_{n}\delta (z)\right) \right] \,\right\vert   \label{eq1}
\\
&\leq &\left\vert F_{\varepsilon}\left( g(z;\beta _{1})-g(z;\beta _{0})-r_{n}\delta
(z)\mid z\right) -F_{\varepsilon}\left( g(z;\beta _{2})-g(z;\beta _{0})-r_{n}\delta
(z)\mid z\right) \right\vert   \notag \\
&\leq &f_{\varepsilon}(v_{n}\mid z)\left\vert g(z;\beta _{1})-g(z;\beta
_{2})\right\vert   \notag \\
&\leq &CA\left( z\right) \left\Vert \beta _{1}-\beta _{2}\right\Vert   \notag
\end{eqnarray}%
for some $v_{n}$ between $g(z;\beta _{1})-g(z;\beta _{0})-r_{n}\delta (z)$
and $g(z;\beta _{2})-g(z;\beta _{0})-r_{n}\delta (z).$ By Pakes and Pollard
(1989, Lemma 2.13), the class of functions $\{z\mapsto \mathbb{E}\left[ \psi
_{\tau }\left( l(\varepsilon,z;\beta )+r_{n}\delta (z)\right) \right] :\beta \in B\}$
is \emph{Euclidean(c,d)} for an envelope with a finite fourth moment, with $c
$ and $d$ independent of $n$. Deduce that the empirical process $\nu
_{n}\left( \beta \right) $, $\beta \in B$, is indexed by a class of
functions that is Euclidean for a squared integrable envelope. Finally,
condition (ii) of Corollary 8 of \cite{Sherman1994}, can be checked from
inequalities like in (\ref{eq1}) and conditions on $\left\vert
\dot{g}(z;\beta )-\dot{g}(z;\beta _{0})\right\vert $.

On the other hand, because $\widehat{\beta }$ minimizes $\Gamma _{n}\left(
\beta \right) $ defined in (\ref{est1}) over $\beta $, the directional
derivative of $\Gamma _{n}\left( \beta \right) $ at $\widehat{\beta }$ along
any direction $\gamma $ (with $\left\Vert \gamma \right\Vert =1$) is
nonnegative.\ That is
\begin{eqnarray}\label{dir_der}
0 &\leq &\lim_{t\rightarrow 0}t^{-1}\left[ \Gamma _{n}(\widehat{\beta }%
+t\gamma )-\Gamma _{n}(\widehat{\beta })\right]  \\
&=&-\sum_{\left\{ Y_{i}\neq g(Z_{i};\widehat{\beta })\right\} }\psi _{\tau
}\left( Y_{i}-g(Z_{i};\widehat{\beta })\right) \gamma \,^{\prime }\dot{g}%
(Z_{i};\widehat{\beta }) \nonumber \\
&&+\lim_{t\rightarrow 0}\sum_{\left\{ Y_{i}=g(Z_{i};\widehat{\beta }%
)\right\} }t^{-1}\rho _{\tau }\left( g(Z_{i};\widehat{\beta })-g(Z_{i};%
\widehat{\beta }+t\gamma )\right) \nonumber \\
&=&-\sum_{\left\{ Y_{i}\neq g(Z_{i};\widehat{\beta })\right\} }\psi _{\tau
}\left( Y_{i}-g(Z_{i};\widehat{\beta })\right) \gamma \,^{\prime }\dot{g}%
(Z_{i};\widehat{\beta }) \nonumber\\
&&-\sum_{\left\{ Y_{i}=g(Z_{i};\widehat{\beta })\right\} }\psi _{\tau
}\left( -\gamma \,^{\prime }\dot{g}(Z_{i};\widehat{\beta })\right) \gamma
\,^{\prime }\dot{g}(Z_{i};\widehat{\beta }) \nonumber\\
&=&-D_{1n}(\widehat{\beta })-D_{2n}(\widehat{\beta }). \nonumber
\end{eqnarray}%
By Assumption \ref{qas2}, $|D_{2n}(\widehat{\beta })|$ is bounded by $%
\sum_{\left\{ Y_{i}=g(Z_{i};\widehat{\beta })\right\} }A(Z_{i}).$ As, for
any $x$, the error term $u$ has a continuous law given $Z=z$, the number of
observations with $Y_{i}=g(Z_{i};\widehat{\beta })$ is bounded in
probability as the sample size tends to infinity. On the other hand, the
moment condition on $A\left( \cdot \right) $ implies that $\max_{1\leq i\leq
n}A(Z_{i})=o_{\mathbb{P}}\left( n^{1/2}\right) .$ As $\gamma $ is an arbitrary
direction, it follows that%
\begin{equation}
\frac{1}{\sqrt{n}}\sum_{i=1}^{n}\psi _{\tau }\left( Y_{i}-g(Z_{i};\widehat{%
\beta })\right) \dot{g}(Z_{i};\widehat{\beta })=o_{\mathbb{P}}\left( 1\right) .
\label{negli2}
\end{equation}%
Finally, since $\widehat{\beta }-\beta _{0}=o_{\mathbb{P}}\left( 1\right) $ and $\tau = F_\varepsilon(0\mid Z_i) $, deduce that%
\begin{eqnarray*}
\nu _{n}\left( \beta _{0}\right)  &=&\nu _{n}(\widehat{\beta })+o_{\mathbb{P}}\left(
1\right) \hskip9cm\text{[by (\ref{negli1})]} \\
&=&-\frac{1}{\sqrt{n}}\sum_{i=1}^{n}\mathbb{E}\left[ \psi _{\tau }\left(
Y_{i}-g(Z_{i};\widehat{\beta })\right) \mid Z_{i}\right] \dot{g}(Z_{i};%
\widehat{\beta })+o_{\mathbb{P}}\left( 1\right) \hskip1.8cm\text{[by (\ref{negli2})]}
\\
&=&\frac{1}{\sqrt{n}}\sum_{i=1}^{n}\left[ F_{\varepsilon}\left( g(\,Z_{i};\widehat{%
\beta }\,)-g(Z_{i};\beta _{0})-r_{n}\delta (\,Z_{i}\,)\mid Z_{i}\right)
-\tau \right] \dot{g}(Z_{i};\widehat{\beta })+o_{\mathbb{P}}\left( 1\right)  \\
&=&\left\{ \frac{1}{n}\sum_{i=1}^{n}f_{\varepsilon}(0\mid Z_{i})\dot{g}(Z_{i};\beta
_{0})\dot{g}^{\prime }(Z_{i};\beta _{0})\right\} \sqrt{n}\left( \widehat{%
\beta }-\beta _{0}\right)  \\
&&-r_{n}\left\{ \frac{1}{\sqrt{n}}\sum_{i=1}^{n}f_{\varepsilon}(0\mid Z_{i})\delta
(\,Z_{i}\,)\dot{g}(Z_{i};\beta _{0})\right\}  \\
&&+r_{n}^{2}\sqrt{n}\left\{ \frac{1}{n}\sum_{i=1}^{n}f_{\varepsilon}^{\,\prime }(0\mid
Z_{i})\delta ^{2}(\,Z_{i}\,)\dot{g}(Z_{i};\beta _{0})\right\}  \\
&&+o_{\mathbb{P}}\left( \sqrt{n}\|\widehat{\beta }-\beta _{0}\|\right) +o_{\mathbb{P}}\left(
r_{n}^{2}\sqrt{n}\right) ,
\end{eqnarray*}%
where the last equality is based on a local expansions of $F_{\varepsilon}\left( \cdot
\mid z\right) $ and $g(z;\cdot ).$ By the law of large numbers, the central
limit theorem and the fact that $\nu _{n}\left( \beta _{0}\right)
=O_{\mathbb{P}}\left( 1\right) $ and the random vector $f_{u}(0\mid Z)\delta (\,Z\,)\dot{g}(Z;\beta
_{0})$ has zero mean, we obtain
\begin{equation*}
\mathbb{E}[f_{\varepsilon}(0\mid Z)\dot{g}(Z;\beta _{0})\dot{g}^{\prime }(Z;\beta
_{0})]\sqrt{n}\left( \widehat{\beta }-\beta _{0}\right) +r_{n}^{2}\sqrt{n}%
\mathbb{E}[f_{\varepsilon}^{\,\prime }(0\mid Z)\delta ^{2}(Z)\dot{g}(Z;\beta
_{0})]=O_{\mathbb{P}}(1)
\end{equation*}%
from which the result follows.
\end{proof}

\quad

Lemma \ref{conv_alt} shows in particular that under $\mathcal{H}_{1n},$ $\widehat{%
\beta }-\beta _{0}=O_{\mathbb{P}}(n^{-1/2}+r_{n}^{2}).$ To our best knowledge, this
result on the behavior of $\widehat{\beta }$ under the local alternatives is
new. \cite{He2003} only considered the case $r_{n}=n^{-1/2}$ while \cite{Zheng1998} 
assumed $\widehat{\beta }-\beta ^{\ast }=O_{\mathbb{P}}(n^{-1/2})$ under $%
\mathcal{H}_{1n}$, for some fixed $\beta ^{\ast }$. Our Lemma \ref{conv_alt} indicates
that such $\sqrt{n}-$convergence assumptions on the local alternatives may
be too restrictive. Below, we improve the point (C) in the Theorem of \cite{Zheng1998} also
because we can take into account the rates of convergence of $\widehat{\beta }$ under the alternatives slower than
$O_{\mathbb{P}}(n^{-1/2})$.

{In the case of a fixed deviation from the null hypothesis, that is
  $r_n\equiv 1,$ the tools used for proving Theorem \ref{l_power}
  could be easily adapted to show the $\sqrt{n}-$convergence of
  $\widehat \beta$ to $\beta^*$ that minimizes the map
  $\beta\mapsto\mathbb{E}[\rho_\tau (Y- g(Z,\beta))] =
  \mathbb{E}[\rho_\tau (g(Z,\beta_0)+\delta(Z)+\varepsilon -
    g(Z,\beta))].$ The consistency of the test is then a consequence
  of the fact that $nh^{1/2}I_n(\beta^*)$ tends to infinity.}

Let $\delta _{i}=\delta (Z_{i})$ and let $G_{i}\left( \beta ,\beta
_{0}\right) $ and $K_{h,ij}$ be defined as in equation (\ref{simp_not}). Under $%
\mathcal{H}_{1n} $%
\begin{eqnarray*}
W_{n}(\beta ) &=&\frac{h^{-1}}{n(n-1)}\sum\limits_{j\neq i}\left[ \mathbb{I}%
\{Y_{i}\leq g(Z_{i};\beta )\}-\tau \right] \left[ \mathbb{I}\{Y_{j}\leq
g(Z_{j};\beta )\}-\tau \right] K_{h,ij} \psi_{ij}\\
&=&\frac{h^{-1}}{n(n-1)}\sum\limits_{j\neq i}\left[ \mathbb{I}\{\varepsilon_{i}\leq
G_{i}(\beta ,\beta _{0})-r_{n}\delta _{i}\}-F_{\varepsilon}\left( 0\mid Z_{i}\right) %
\right] \\
&&\qquad \qquad \qquad \times \left[ \mathbb{I}\{\varepsilon_{j}\leq G_{j}(\beta
,\beta _{0})-r_{n}\delta _{j}\}-F_{\varepsilon}\left( 0\mid Z_{j}\right) \right]
K_{h,ij}\psi_{ij}.
\end{eqnarray*}%
Let us decompose
\begin{eqnarray*}
F_{\varepsilon}\left( 0\mid Z_{i}\right) &=&F_{\varepsilon}\left( G_{i}(\beta ,\beta
_{0})-r_{n}\delta _{i}\mid Z_{i}\right) -f_{\varepsilon}\left( 0\mid Z_{i}\right)
\left\{ \dot{g}^{\prime }(Z_{i};\beta _{0})\left( \beta -\beta _{0}\right)
-r_{n}\delta _{i}\right\} \\
&&\qquad -2^{-1}r_{n}^{2}f_{\varepsilon}^{\,\prime }\left( 0\mid Z_{i}\right) \delta
_{i}^{2}+O_{\mathbb{P}}\left( \left\Vert \beta -\beta _{0}\right\Vert
^{2}+r_{n}\left\Vert \beta -\beta _{0}\right\Vert \right) +o_{\mathbb{P}}\left(
r_{n}^{2}\right) .
\end{eqnarray*}%
We can write%
\begin{equation*}
W_{n}(\beta )=W_{1n}(\beta )+2[W_{2n}(\beta )+W_{3n}(\beta )+W_{4n}(\beta
)]+W_{5n}(\beta )+2W_{6n}(\beta )+W_{7n}+R_{n}
\end{equation*}%
where
\begin{multline*}
W_{1n}(\beta )=\frac{h^{-1}}{n(n-1)}\!\sum\limits_{j\neq i}\left[ \mathbb{I}%
\{\varepsilon_{i}\leq G_{i}(\beta ,\beta _{0})-r_{n}\delta _{i}\}-F_{\varepsilon}\left(
G_{i}(\beta ,\beta _{0})-r_{n}\delta _{i}\mid Z_{i}\right) \right] \\
\times \left[ \mathbb{I}\{\varepsilon_{j}\leq G_{j}(\beta ,\beta _{0})-r_{n}\delta
_{j}\}-F_{\varepsilon}\left( G_{j}(\beta ,\beta _{0})-r_{n}\delta _{j}\mid
Z_{j}\right) \right] K_{h,ij}\psi_{ij}
\end{multline*}%
$W_{2n}(\beta )=\left( \beta -\beta _{0}\right) ^{\prime }\widetilde{W}%
_{2n}(\beta )$ with
\begin{eqnarray*}
\widetilde{W}_{2n}(\beta ) &=&\frac{h^{-1}}{n(n-1)}\sum\limits_{j\neq i}%
\left[ \mathbb{I}\{\varepsilon_{i}\leq G_{i}(\beta ,\beta _{0})-r_{n}\delta
_{i}\}-F_{\varepsilon}\left( G_{i}(\beta ,\beta _{0})-r_{n}\delta _{i}\mid
Z_{i}\right) \right] \\
&&\qquad \qquad \qquad \times f_{\varepsilon}\left( 0\mid Z_{j}\right) \dot{g}%
(Z_{j};\beta _{0})K_{h,ij} \psi_{ij},
\end{eqnarray*}
\begin{eqnarray*}
W_{3n}(\beta ) &=&\frac{r_{n}h^{-1}}{n(n-1)}\sum\limits_{j\neq i}\left[
\mathbb{I}\{\varepsilon_{i}\leq G_{i}(\beta ,\beta _{0})-r_{n}\delta
_{i}\}-F_{\varepsilon}\left( G_{i}(\beta ,\beta _{0})-r_{n}\delta _{i}\mid
Z_{i}\right) \right] \\
&&\qquad \qquad \qquad \times f_{\varepsilon}\left( 0\mid Z_{j}\right) \delta
_{j}K_{h,ij} \psi_{ij},
\end{eqnarray*}%
\begin{eqnarray*}
W_{4n}(\beta ) &=&\frac{r_{n}^{2}h^{-1}}{2n(n-1)}\sum\limits_{j\neq i}\left[
\mathbb{I}\{\varepsilon_{i}\leq G_{i}(\beta ,\beta _{0})-r_{n}\delta
_{i}\}-F_{\varepsilon}\left( G_{i}(\beta ,\beta _{0})-r_{n}\delta _{i}\mid
Z_{i}\right) \right] \\
&&\qquad \qquad \qquad \times f_{\varepsilon}^{\,\prime }\left( 0\mid Z_{j}\right)
\delta _{j}^{2}K_{h,ij}  \psi_{ij},
\end{eqnarray*}%
$W_{5n}(\beta )=\left( \beta -\beta _{0}\right) ^{\prime }\widetilde{W}%
_{5n}\left( \beta -\beta _{0}\right) $ with
\begin{equation*}
\widetilde{W}_{5n}=\frac{h^{-1}}{n(n-1)}\sum\limits_{j\neq i}f_{\varepsilon}\left(
0\mid Z_{i}\right) \dot{g}(Z_{i};\beta _{0})\dot{g}^{\prime }(Z_{j};\beta
_{0})f_{\varepsilon}\left( 0\mid Z_{j}\right) K_{h,ij}  \psi_{ij}=O_{\mathbb{P}}(1),
\end{equation*}%
$W_{6n}(\beta )=\left( \beta -\beta _{0}\right) ^{\prime }\widetilde{W}_{6n}$
with
\begin{equation*}
\widetilde{W}_{6n}=\frac{r_{n}h^{-1}}{n(n-1)}\sum\limits_{j\neq
i}f_{\varepsilon}\left( 0\mid Z_{i}\right) \delta _{i}f_{\varepsilon}\left( 0\mid Z_{j}\right)
\dot{g}(X_{j};\beta _{0})K_{h,ij} \psi_{ij}=O_{\mathbb{P}}(r_{n}),
\end{equation*}%
\begin{equation*}
W_{7n}=\frac{r_{n}^{2}h^{-1}}{n(n-1)}\sum\limits_{j\neq i}f_{\varepsilon}\left( 0\mid
Z_{i}\right) \delta (X_{i})f_{\varepsilon}\left( 0\mid Z_{j}\right) \delta
(Z_{j})K_{h,ij}  \psi_{ij} =C_{1}r_{n}^{2}+o_{\mathbb{P}}(r_{n}^{2})
\end{equation*}%
with $C_{1}>0$ and $R_{n}$ a reminder term that is negligible because of the
properties of $f_{\varepsilon}^{\,\prime }$ and $\dot{g}$. Note that the $U-$%
statistics $\widetilde{W}_{5n}$, $\widetilde{W}_{6n}$ and $W_{7n}$ depend
only on the $X_{i}.$ Their orders are obtained from elementary calculations of mean and variance.

Next, we can write $W_{1n}(\beta )=\left\{ W_{1n}(\beta )-W_{1n}(\beta
_{0})\right\} +W_{1n}(\beta _{0}).$ As $W_{1n}(\beta _{0})$ is centered, its
order in probability is given by the variance. We have
\begin{eqnarray*}
\mbox{\rm Var}(W_{1n}(\beta _{0})\mid Z_{1},...,Z_{n}) &=&\frac{1}{%
n^{2}(n-1)^{2}}\sum_{i\neq j}F_{\varepsilon}\left( -r_{n}\delta _{i}\mid Z_{i}\right)
[1-F_{\varepsilon}\left( -r_{n}\delta _{i}\mid Z_{i}\right) ] \\
&&\,\,\,\,\,\times F_{\varepsilon}\left( -r_{n}\delta _{j}\mid Z_{j}\right)
[1-F_{\varepsilon}\left( -r_{n}\delta _{j}\mid Z_{j}\right) ]h^{-2}K_{h,ij}^{2} \psi_{ij}\left(
\mu \right)  \\
&\leq &\frac{h^{-1}}{16n(n-1)}\left[\frac{1}{n(n-1)} \sum_{i\neq j}h^{-1}K_{h,ij}^{2}  \psi_{ij}\right]
\end{eqnarray*}%
The expectation of the last $U-$statistic in the display converges to a constant  while the variance tends to zero. As $W_{1n}(\beta _{0})$ is of zero
conditional mean given the $Z_{i}$, deduce that the variance of
$W_{1n}(\beta _{0})$ is bounded by $Cn^{-2}h^{-1}$. By Chebyshev's
inequality, $W_{1n}(\beta _{0})=o_{\mathbb{P}}\left( r_{n}^{2}\right) $, provided
that $r_{n}^{2}nh^{1/2}\rightarrow \infty .$ Next, let
\begin{multline*}
H_{1n}(Z_i ,Z_j, \beta ) =  \left[ \mathbb{I}
\{\varepsilon_{i}\leq G_{i}(\beta ,\beta _{0})-r_{n}\delta _{i}\}-F_{\varepsilon}\left(
G_{i}(\beta ,\beta _{0})-r_{n}\delta _{i}\mid Z_{i}\right) \right] \\
\times \left[ \mathbb{I}\{\varepsilon_{j}\leq G_{j}(\beta ,\beta _{0})-r_{n}\delta
_{j}\}-F_{\varepsilon}\left( G_{j}(\beta ,\beta _{0})-r_{n}\delta _{j}\mid
Z_{j}\right) \right]   K_{h,ij}  \psi_{ij}, \qquad \beta\in B.
\end{multline*}
By the arguments used for Lemma \ref{conv_alt} above, the class of functions
$\{H_{1n}(\cdot ,\cdot, \beta ) : \beta\in B \}$ is \emph{Euclidean(c,d)}
for an envelope with a finite fourth moment, with $c$ and $d$ independent of $n$.
Now, we can use equation (A.11) of \cite{Zheng1998} and his Lemma 1
with the condition (ii) replaced by $\mathbb{E}[H_{1n}(\cdot ,\beta
)-H_{1n}(\cdot ,\beta _{0})]^{2}\leq \Lambda \left\Vert \beta -\beta
_{0}\right\Vert $. By a close inspection of the proof
of Zheng's Lemma 1, see his equations (A.2) to (A.5), it is obvious to
adapt his conclusion and to deduce that in our setup for any $0<\alpha <1$
\begin{equation*}
W_{1n}(\beta )-W_{1n}(\beta _{0})=O_{\mathbb{P}}\left( n^{-1}h^{-1}\left\Vert \beta
-\beta _{0}\right\Vert ^{\alpha /2}\right) =O_{\mathbb{P}}\left( n^{-1}h^{-1}\left\{
r_{n}+n^{-1/4}\right\} ^{\alpha }\right)
\end{equation*}%
uniformly over $O_{\mathbb{P}}(r_{n}^{2}+n^{-1/2})$ neighborhoods of $\beta _{0}.$
Thus, when $n^{1/2}r_{n}^{2}\rightarrow \infty ,$ we have
\begin{equation*}
W_{1n}(\widehat{\beta })-W_{1n}(\beta _{0})=O_{\mathbb{P}}\left(
n^{-1}h^{-1}r_{n}^{\alpha }\right) =O_{\mathbb{P}}\left( n^{-1/2}\right) =o_{\mathbb{P}}\left(
r_{n}^{2}\right) ,
\end{equation*}%
whereas in the case where $n^{1/2}r_{n}^{2}$ is bounded, use $nh^{1/2}r_{n}^{2}%
\rightarrow \infty $ and take $\alpha $ sufficiently close to one to obtain
\begin{equation*}
W_{1n}(\widehat{\beta })-W_{1n}(\beta _{0})=O_{\mathbb{P}}\left( n^{-1-\alpha
/4}h^{-1}\right) =o_{\mathbb{P}}\left( r_{n}^{2}\right) .
\end{equation*}%

The remaining terms $W_{2n}$, $W_{3n}$ and $W_{4n}$ can be treated in the
following way. By Hoeffding's decomposition
\begin{equation*}
r_{n}^{-1}hW_{3n}(\beta )=U_{n}^{2}(\beta )+U_{n}^{1}(\beta )
\end{equation*}%
with $U_{n}^{1}$, $U_{n}^{2}$ degenerate $U-$processes or order 1 and 2,
respectively. In view of Assumption \ref{qas2}(d) and the fact that $K\left(
\cdot \right) $ is bounded, apply Corollary 4 of \cite{Sherman1994} to deduce
that $U_{n}^{2}(\beta )=O_{\mathbb{P}}\left( n^{-1}\right) $ uniformly in $\beta .$
If $K_{h,ij}\left( \theta \right) =K_{h}((X_{i}-X_{j})^{\prime }\theta )$
and
\begin{equation*}
\xi \left( Z_{i} \right) =\mathbb{E}\left[ \mathbb{E}\left\{
f_{\varepsilon}\left( 0\mid Z_{j}\right) \delta \left( Z_{j}\right) \mid Z_{j}^{\prime
}\theta \right\} h^{-3/4}K_{h,ij} \psi_{ij} \mid Z_{i}\right]
\end{equation*}%
we can write
\begin{equation*}
U_{n}^{1}(\beta )=\frac{h^{3/4}}{n}\!\sum\limits_{i}\left[ \mathbb{I}%
\{\varepsilon_{i}\leq G_{i}(\beta ,\beta _{0})-r_{n}\delta _{i}\}-F_{\varepsilon}\left(
G_{i}(\beta ,\beta _{0})\!-\!r_{n}\delta _{i}\mid Z_{i}\right) \right]
\xi \left( Z_{i} \right) .
\end{equation*}%
By H\"{o}lder inequality, Assumption \ref{qas1}(c) and a change of
variables,
$$
\left\vert \xi \left( Z_{i} \right) \right\vert  \leq \mathbb{E}^{1/4}\left[ \delta^{4}( Z_{j} )\right] \mathbb{E}^{3/4}\left[
h^{-1}K_{h,ij}^{4/3}\mid Z_{i} \right]\leq C,
$$
for some $C>0.$ Now, by Corollary 4 of \cite{Sherman1994}, $h^{-3/4}U_{n}^{1}(\beta
)=O_{\mathbb{P}}\left( n^{-1/2}\right) $ uniformly in $\beta .$ As $%
nh^{1/2}r_{n}^{2}\rightarrow \infty ,$ deduce that
\begin{equation*}
\sup_{\beta }|W_{3n}(\beta )|=O_{\mathbb{P}}\left(
r_{n}h^{-1}n^{-1}+r_{n}h^{-1/4}n^{-1/2}\right) =o_{\mathbb{P}}(r_{n}^{2}).
\end{equation*}%
By similar arguments, $\sup_{\beta }\left\vert W_{4n}(\beta )\right\vert
=o_{\mathbb{P}}(r_{n}^{2})$ (here apply H\"{o}lder inequality with $p=q=2$) and $W_{3n}$, $\sup_{\beta }|\widetilde{W}_{2n}(\beta
)|=O_{\mathbb{P}}\left( h^{-1}n^{-1}+h^{-1/4}n^{-1/2}\right) $, and thus
\begin{equation*}
\sup_{\beta }|W_{2n}(\beta )|=O_{\mathbb{P}}(r_{n}^{2}+n^{-1/2})O_{\mathbb{P}}\left(
h^{-1}n^{-1}+h^{-1/4}n^{-1/2}\right) =o_{\mathbb{P}}(r_{n}^{2}).
\end{equation*}%
Collecting results, under $\mathcal{H}_{1n}$, $T_{n}\geq
Cnh^{1/2}r_{n}^{2}\{1+o_{\mathbb{P}}(1)\}$ or some constants $C>0$. Now, the proof is
complete.


\subsection{Proof of Theorem \protect\ref{as_law_boot}}
Let $W_{n}^{\ast }(\beta )$ be the statistic obtained after replacing $%
U_{i}\left( \beta \right) $ with $U_{i}^*\left( \beta \right)=\mathbb{I}\{Y_{i}^{\ast }\leq g(Z_{i};\beta
)\}-\tau $ in the formula of $W_{n}(\beta ).$ The proof of the bootstrap
procedure consistency follows the steps of the proof of Theorem \ref{as_law}, but requires several specific ingredients: (a) the
convergence in law of $nh^{1/2}W_{n}^{\ast }(\widehat{\beta })/v_{n}$ \emph{conditionally} upon the original sample; and (b)
the $O_{\mathbb{P}}\left(
n^{-1/2}\right) $ rate for $\widehat{\beta }^{\ast }-\widehat{\beta }$, and
the negligibility of $W_{n}^{\ast }(\widehat{\beta }^{\ast })-W_{n}^{\ast }(%
\widehat{\beta })$ given the original sample.
If $S_{1n}^{*}$ and $S_{2n}^{*}$ denote bootstrapped statistics, $S_{1n}^{*}$ is bounded in probability given the sample  if
$$
\lim_{M\rightarrow \infty} \mathbb{P}[|S_{1n}^{*}|>M \mid Y_1,Z_1,\cdots, Y_n,Z_n] = o_p(1).
$$
while  $S_{2n}^{*}$ is asymptotically negligible given the sample  if
$$
\forall \epsilon >0, \qquad \mathbb{P}[|S_{2n}^{*}|>\epsilon \mid Y_1,Z_1,\cdots, Y_n,Z_n] = o_p(1).
$$

The asymptotic normality of $ nh^{1/2}W_{n}^{\ast }(%
\widehat{\beta })/v_{n}$ given the sample is obtained below from a martingale central
limit theorem as stated in \cite{Hall1980}.

\begin{lemma}
\label{hall_heyde}
Under the assumptions of Theorem \ref{as_law_boot},
\begin{equation*}
\sup_{t\in \mathbb{R}}\left\vert \mathbb{P}\left( nh^{1/2}W_{n}^{\ast }(%
\widehat{\beta })/v_{n}\leq t\mid Y_{1},Z_{1},...,Y_{n},Z_{n}\right) -\Phi
(t)\right\vert \rightarrow 0,\qquad \text{in probability}.
\end{equation*}%
\end{lemma}

\medskip

\begin{proof} The proof is based on  the Central limit Theorem (CLT) for martingale arrays, see Corollary 3.1 of \cite{Hall1980}.
Recall that $U_{i}^\ast (\widehat{\beta })= \mathbb{I}\{Y_{i}^{\ast
}\leq g(Z_{i};\widehat{\beta }^{\ast })\}-\tau $. Define the martingale array $\left\{ S_{n,m}^{*},\,\mathcal{F}_{n,m}^{*},\,1\leq m\leq n,\, n\geq1\right\} $
where $S_{n,1}^{*}=0$ and $S_{n,m}^{*}=\sum_{i=2}^{m}G_{n,i}^{*}$
with
\[
G_{n,i}^{*}=\dfrac{2h^{-1/2}}{n-1}U_{i}^{*}(\widehat{\beta})\sum_{j=1}^{i-1}
U_{j}^{*}(\widehat{\beta})K_{h,ij}\psi_{ij},
\]
and $\mathcal{F}_{n,m}^{*}$ is the $\sigma$-field generated by $\left\{ \overline{Z},\eta_{1},\dots,\eta_{m}\right\} $ where
$\overline{Z} = \left\{Y_{1},\dots,Y_{n}, Z_{1},\dots,Z_{n}\right\}$.
Thus $nh^{1/2}W_{n}^{*}(\widehat{\beta})=S_{n,n}^{*}$. Next define
\begin{eqnarray*}
V_{n}^{2*} & = & \sum_{i=2}^{n}\mathbb{E}\left[G_{n,i}^{2*}\mid\mathcal{F}_{n,i-1}^{*}\right]\\
 & = & \dfrac{4h^{-1}\tau(1-\tau)}{(n-1)^{2}}\sum_{i=2}^{n}\sum_{j=1}^{i-1}\sum_{k=1}^{i-1}U_{j}^{*}(\widehat{\beta})U_{k}^{*}(\widehat{\beta})K_{h,ij}K_{h,ik}\psi_{ij}\psi_{ik}\\
 & = & \dfrac{4h^{-1}\tau(1-\tau)}{(n-1)^{2}}\sum_{i=2}^{n}\sum_{j=1}^{i-1}U_{j}^{*2}(\widehat{\beta})K_{h,ij}^{2}\psi_{ij}^{2}\\
 &  & +\dfrac{8h^{-1}\tau(1-\tau)}{(n-1)^{2}}\sum_{i=3}^{n}\sum_{j=2}^{i-1}\sum_{k=1}^{j-1}U_{j}^{*}(\widehat{\beta})U_{k}^{*}(\widehat{\beta})K_{h,ij}K_{h,ik}\psi_{ij}\psi_{ik}\\
 & = & A_{n}^{*}+B_{n}^{*}.
\end{eqnarray*}
Recall that
$$
v_{n}^{2}=\frac{2 h^{-1}\,\tau ^{2}(1-\tau )^{2}%
}{n(n-1)}\sum\limits_{j\neq i}K_{h,ij}^2 \psi^2_{ij}
$$
and by standard calculations of the means and variance it could be shown to tend to a positive constant. Next,
note that
$$
\mathbb{E}\left[A_{n}^{*}\mid \overline{Z} \right] = \dfrac{4h^{-1}\tau(1-\tau)}{(n-1)^{2}}\sum_{i=2}^{n}\sum_{j=1}^{i-1}
\mathbb{E}\left[U_{j}^{*2}(\widehat{\beta})\mid\overline{Z}\right]K_{h,ij}^{2}\psi_{ij}^{2}
= \dfrac{n }{n-1}\;v_{n}^{2} .
$$
Moreover,
\begin{eqnarray*}
\mathbb{E}\left[\mbox{Var}\left(A_{n}^{*}\mid \overline{Z}\right)\right] & = & \dfrac{16\tau^{2}(1-\tau)^{2}}{h^{2}(n-1)^{4}}\\
&& \;\; \times \sum_{i=2}^{n}\sum_{i^{\prime}=2}^{n}\sum_{j=1}^{i\wedge i^{\prime}-1}\mathbb{E}\left[\mathbb{E}\left[U_{j}^{*4}(\widehat{\beta})-\tau^{2}(1-\tau)^{2}\vert\overline{Z}\right]K_{h,ij}^{2}K_{h,i^{\prime}j}^{2}\psi_{ij}^{2}\psi_{i^{\prime}j}^{2}\right]\\
 & = & \dfrac{16\tau^{4}(1-\tau)^{4}\{\tau(1-\tau)(1-3\tau(1-\tau))-1\}}{h^{2}(n-1)^{4}} \\
&& \qquad \qquad \qquad \qquad \qquad \qquad\;\; \times
 \sum_{i=2}^{n}\sum_{i^{\prime}=2}^{n}\sum_{j=1}^{i\wedge i^{\prime}-1}\mathbb{E}\left[K_{h,ij}^{2}K_{h,i^{\prime}j}^{2}\psi_{ij}^{2}\psi_{i^{\prime}j}^{2}\right]\\
 & = & \dfrac{32\tau^{4}(1-\tau)^{4}(\tau(1-\tau)(1-3\tau(1-\tau))-1)}{h^{2}(n-1)^{4}}\\
&& \qquad \qquad\qquad \qquad \qquad \qquad\;\; \times \sum_{i=3}^{n}\sum_{i^{\prime}=2}^{i-1}\sum_{j=1}^{i^{\prime}-1}\mathbb{E}\left[K_{h,ij}^{2}K_{h,i^{\prime}j}^{2}\psi_{ij}^{2}\psi_{i^{\prime}j}^{2}\right]\\
 &  & +\dfrac{16\tau^{4}(1-\tau)^{4}(\tau(1-\tau)(1-3\tau(1-\tau))-1)}{h^{2}(n-1)^{4}}\sum_{i=2}^{n}\sum_{j=1}^{i-1}\mathbb{E}\left[K_{h,ij}^{4}\psi_{ij}^{4}\right]\\
 & = & O(n^{-1})+O(n^{-2}h^{-1})
\end{eqnarray*}
because $\psi_{ij}$, $\mathbb{E}\left[h^{-1}K_{h,ij}^{4}\right]$
and $\mathbb{E}\left[h^{-2}K_{h,ij}^{2}K_{h,i^{\prime}j}^{2}\right]$
are bounded for all pairwise
distinct indexes $i$, $i^{\prime}$ and $j$. Deduce that
 $A_{n}^{*}/v_{n}^{2}\rightarrow 1$ in probability.
On the other hand,
$$
\mathbb{E}\left[B_{n}^{*2}\right] = \dfrac{8\tau^{4}(1-\tau)^{4}}{h^{2}(n-1)^{4}}\sum_{i=3}^{n}\sum_{j=2}^{i-1}\sum_{k=1}^{j-1}\mathbb{E}\left[K_{h,ij}^{2}K_{h,ik}^{2}\psi_{ij}^{2}\psi_{ik}^{2}\right]=  O(n^{-1})
$$
so that $V_{n}^{2*}/v_{n}^{2}\rightarrow 1$ in probability.
To use the CLT it remains to check the Lindeberg condition. For any $\epsilon>0$,
\begin{multline*}
\mathbb{E}\left[ \sum_{i=2}^{n}\mathbb{E}\left[G_{n,i}^{*2}\mathbb{I}(G_{n,i}^{*2}>\epsilon)\mid\mathcal{F}_{n,i-1}^{*}\right] \right]
\leq \epsilon^{-4}\mathbb{E}\left[ \sum_{i=2}^{n}\mathbb{E}\left[G_{n,i}^{*4}\mid\mathcal{F}_{n,i-1}^{*}\right]\right]\\
\leq  \dfrac{16\tau^{3}(1-\tau)^{3}\{1-3\tau(1-\tau)\}}{\epsilon^{4}h^{2}(n-1)^{4}}\sum_{i=2}^{n}\sum_{j=1}^{i-1}\sum_{k=1}^{i-1}\mathbb{E}\left[K_{h,ij}^{2}K_{h,ik}^{2}\psi_{ij}^{2}\psi_{ik}^{2}\right]\\
\leq  \dfrac{32\tau^{3}(1-\tau)^{3}\{1-3\tau(1-\tau)\}}{\epsilon^{4}h^{2}(n-1)^{4}}\sum_{i=2}^{n}\sum_{j=1}^{i-1}\sum_{k=1}^{j-1}\mathbb{E}\left[K_{h,ij}^{2}K_{h,ik}^{2}\psi_{ij}^{2}\psi_{ik}^{2}\right]\\
+\dfrac{16\tau^{3}(1-\tau)^{3}\{1-3\tau(1-\tau)\}}{\epsilon^{4}h^{2}(n-1)^{4}}\sum_{i=2}^{n}\sum_{j=1}^{i-1}\mathbb{E}\left[K_{h,ij}^{4}\psi_{ij}^{4}\right]\\
=  O(n^{-1})+O(n^{-2}h).
\end{multline*}
Eventually, applying the CLT for martingale arrays along the subsequences of $V_{n}^{2*}$ that converge almost surely to the limit of $v_n^2$ and subsequences for which the Lindeberg condition is satisfied almost surely,
the result follows. \end{proof}

\qquad

To obtain
the $O_{\mathbb{P}}\left(
n^{-1/2}\right) $ rate for $\widehat{\beta }^{\ast }-\widehat{\beta }$, and
the negligibility of $W_{n}^{\ast }(\widehat{\beta }^{\ast })-W_{n}^{\ast }(%
\widehat{\beta })$ given the original sample, we use a conditional version of the moment
inequality for $U-$processes proved by \cite{Sherman1994}.  Before stating this new result that has its own interest  let us introduce some more
notation: for $k$ a positive integer let $(n)_{k}=n(n-1)...(n-k+1)$ and let $%
\mathbf{i_{k}^{n}}=(i_{1},...,i_{k})$ be a $k-$tuple of distinct integers
from the set $\{1,...,n\}$. Similarly, $\mathbf{i_{k}^{2n}}=(i_{1},...,i_{k})
$ denotes a $k-$tuples of distinct integers from $\{1,...,2n\}$. Moreover, a
function $g$ on $\mathcal{S}^{k}$ is called degenerate if for each $i=1,...,k
$, and all $s_{1},...,s_{i-1},s_{i+1},...,s_{k}\in \mathcal{S}$, $\mathbb{E}%
[g(s_{1},...,s_{i-1},S,s_{i+1},...,s_{k})]=0$.

\begin{lemma}
\label{sher} Let $k$ be a positive integer and $\mathcal{G}$ a degenerate
class of real-valued functions on $\mathbb{R}^{1+q}\times ...\times\mathbb{R}%
^{1+q}$. Suppose $\mathcal{G}$ is \emph{Euclidean(c,d)} for a squared integrable
envelope and some $c,d>0$. Fix $z_1,...,z_n \in \mathbb{R}^{q}$ and let $%
u_1, ...,u_n, u_{n+1}, ...,u_{2n}$ be independent copies of the random
variable $u$. For $i=1,...,n$, let $v_i = (u_i,z_i)$ and $v_{n+i} =
(u_{n+i},z_i)$. Define $g_{\mathbf{i_k^n}}(u_{i_1}, \ldots , u_{i_k}) =
g(v_{i_1}, \ldots , v_{i_k})$ and define $g_{\mathbf{i_k^{2n}}}$ similarly.
Suppose that for any $k-$tuple $\mathbf{i_k^n}$, the function $g_{\mathbf{%
i_k^n}}$ is degenerate as a function of $u_i$ variables (necessarily the
same property holds also for any $k-$tuple $\mathbf{i_k^{2n}}$). Let
\begin{equation*}
U^k_{n, z_1,...,z_n} (g)= (n)_k^{-1}\sum_{\mathbf{i_k^n }} g_{\mathbf{i_k^n}%
}(u_{i_1}, \ldots , u_{i_k}), \quad U^k_{2n, z_1,...,z_n} (g)=
(2n)_k^{-1}\sum_{\mathbf{i_k^{2n}}} g_{\mathbf{i_k^{2n}}}(u_{i_1}, \ldots ,
u_{i_k}).
\end{equation*}
Then for any $\alpha\in (0,1)$, there exists a constant $\Lambda$ depending
only on $\alpha$ and $k$ (and independent of $n$ and the sequence $%
z_1,...,z_n$) such that
\begin{equation*}
\mathbb{E} \left[ \sup_{\mathcal{G}}| n^{k/2} U^k_{n, z_1,...,z_n}(g) | %
\right] \leq \Lambda \mathbb{E}^{1/2} \left[ \sup_{\mathcal{G}} \{ U^k_{2n,
z_1,...,z_n}(g^2)\}^\alpha \right] .
\end{equation*}
\end{lemma}

\medskip

\begin{proof}
We sketch the steps of the proof that follows the lines of the proof of the
Main Corollary in \cite{Sherman1994}. For the sake of
simplicity, we only consider the case of Euclidean families
for a constant envelope. Fix $n$ and $z_1,...,z_n$ arbitrarily.

\emph{i) Symmetrization inequality.} For each $g\in \mathcal{G}$ define $%
\widetilde{g}(\mathbf{i_{k}^{n}})$ as a sum of $2^{k}$ terms, each having
the form
\begin{equation*}
(-1)^{r}g_{\mathbf{i_{k}^{n}}}(u_{i_{1}}^{\ast },\ldots ,u_{i_{k}}^{\ast })
\end{equation*}%
with $u_{i_{j}}^{\ast }$ equal to either $u_{i_{j}}$ or $u_{n+i_{j}}$ where $%
i_{j}$ ranges over the set $\{1,...,n\}$, and $r$ is the number of elements $%
u_{i_{1}}^{\ast },...,u_{i_{k}}^{\ast }$ belonging to $\{u_{n+1},...,u_{2n}\}
$. Independently, take a sample $\sigma _{1},...,\sigma _{n}$ of Rademacher
random variables, that is symmetric variables on the two points set $\{-1,1\}
$. Let $\Phi $ be a convex function on $[0,\infty )$. Then
\begin{equation}
\mathbb{E}\Phi \left( \sup_{\mathcal{G}}\left\vert \sum_{\mathbf{i_{k}^{n}}%
}g_{\mathbf{i_{k}^{n}}}(u_{i_{1}},\ldots ,u_{i_{k}})\right\vert \right) \leq
\mathbb{E}\Phi \left( \sup_{\mathcal{G}}\left\vert \sum_{\mathbf{i_{k}^{n}}%
}\sigma _{i_{1}}\ldots \sigma _{i_{k}}\widetilde{g}(\mathbf{i_{k}^{n}}%
)\right\vert \right) .  \label{symm1}
\end{equation}%
The proof of this inequality is omitted as it can be derived with only
formal changes from the proof of \cite{Sherman1994}'s symmetrization inequality.
It can be also be derived from the lines of \cite{Pena1999}, 
Theorem 3.5.3 (see also Remark 3.5.4 of de la Pe\~{n}a and Gin\'{e}).

\emph{ii) Maximal inequality.} The following arguments are similar to those
in \cite{Sherman1994}, section 5. Define the stochastic process
\begin{equation*}
Z(g)=n^{k/2}\sum_{\mathbf{i_{k}^{n}}}\sigma _{i_{1}}\ldots \sigma _{i_{k}}%
\widetilde{g}(\mathbf{i_{k}^{n}}),\quad g\in \mathcal{G}
\end{equation*}%
and the pseudo-metric $%
d_{U_{2n}^{k}}(g_{1},g_{2})=[U_{2n,z_{1},...,z_{n}}^{k}(|g_{1}-g_{2}|^{2})]^{1/2}
$. Finally, let us remark that for each $g$, by Cauchy-Schwarz inequality
and the definitions of $\widetilde{g}(\mathbf{i_{k}^{n}})$ and $g_{\mathbf{%
i_{k}^{2n}}}$ we have
\begin{equation*}
\sum_{\mathbf{i_{k}^{n}}}\widetilde{g}(\mathbf{i_{k}^{n}})^{2}\leq
2^{k}\sum_{\mathbf{i_{k}^{2n}}}g_{\mathbf{i_{k}^{2n}}%
}^{2}(u_{i_{1}},...,u_{i_{k}})=2^{k}(2n)_{k}U_{2n,z_{1},...,z_{n}}^{k}(g^{2})
\end{equation*}%
which is the counterpart of inequality (5) of \cite{Sherman1994}. Now, we have
all the ingredients to continue exactly as in the proof of Sherman's maximal
inequality and to deduce that for any positive integer $m$
\begin{equation*}
\mathbb{E}\left[ \sup_{\mathcal{G}}|n^{k/2}U_{n,z_{1},...,z_{n}}^{k}(g)|%
\right] \leq \Gamma \mathbb{E}\left[ \int_{0}^{\delta
_{n}^{k}}[D(x,d_{U_{2n}^{k}},\mathcal{G})]^{1/2m}dx\right]
\end{equation*}%
where $D(\epsilon ,d_{U_{2n}^{k}},\mathcal{G})$ are the \emph{packing numbers%
} of the set $\mathcal{G}$ with respect to the pseudometric $d_{U_{2n}^{k}}$%
, $\delta _{n}^{k}=\sup_{\mathcal{G}}\sqrt{U_{2n,z_{1},...,z_{n}}^{k}(g^{2})}
$ and $\Gamma $ is a constant depending only on $m$ and $k$.

\emph{iii) Moment inequality for Euclidean families.} If $\mathcal{G}$ is
\emph{Euclidean(c,d)} for a constant envelope equal to one, then the packing
number $D(\epsilon ,d_{U_{2n}^{k}},\mathcal{G})$ is bounded by $c\epsilon
^{-d}$. To check this, apply the definition of an Euclidean family for $%
\mathcal{G}$ with $\mu $ the measure that places mass $(2n)_{k}^{-1}$ at
each of the $(2n)_{k}$ pairs $(v_{i},v_{j})$, $1\leq i\neq j\leq 2n$.
Finally, our result follows using the arguments of the Main Corollary of
\cite{Sherman1994}.
\end{proof}

\qquad

To establish the rate of  $\widehat{\beta }^{\ast }-\widehat{\beta }$ given the sample, it suffices to consider a simplified version of our Lemma \ref{conv_alt}. By Lemma \ref{sher}, $\sup_{\beta }\left\vert
n^{-1}\Gamma _{n}^*\left( \beta \right) -\mathbb{E}\left[ \rho _{\tau
}(Y-g(Z ;\beta ))\mid \overline{Z} \right] \right\vert$ is asymptotically negligible given the sample $\overline{Z} = \left\{Y_{1},\dots,Y_{n}, Z_{1},\dots,Z_{n}\right\}$. Reconsidering the arguments for the consistency of argmax estimators along almost surely convergent subsequences depending on $\overline{Z}$, deduce that $\widehat\beta^* - \widehat \beta$ is a asymptotically negligible given the sample $\overline{Z}.$ Next, define the empirical process
$$
\nu ^*_{n}\left( \beta \right) = \frac{1}{\sqrt{n}}\sum_{i=1}^{n}\left\{ \psi
_{\tau }\left( Y_{i}^* - g(Z_{i};\beta ) \right) - \mathbb{E} [ \psi
_{\tau }\left( Y_{i} - g(Z_{i};\beta ) \right) \mid \overline{Z} ] \right\} \dot{g}(Z_{i};\beta )
$$
indexed by $\beta .$ Lemma \ref{sher} guarantees that $\sup_{\beta} |\nu ^*_{n}\left( \beta \right)|,$ and in particular $\nu ^*_{n}( \widehat \beta^* )- \nu ^*_{n}( \widehat \beta )$, are bounded in probability given the sample. Proceeding like in (\ref{dir_der}), that is using the directional
derivative of $\Gamma _{n}^*\left( \beta \right) $ at $\widehat{\beta }^*$ along
any direction $\gamma ,$ deduce
\begin{equation*}
\frac{1}{\sqrt{n}}\sum_{i=1}^{n}\psi _{\tau }\left( Y_{i}^*-g(Z_{i};\widehat{%
\beta }^*)\right) \dot{g}(Z_{i};\widehat{\beta }^*)
\end{equation*}
is bounded in probability given the sample (conditional negligibility could be also derived but boundedness given the sample suffices for the present purpose). Since for all $i,$
$$
\mathbb{E}\left[ \psi _{\tau }\left(
Y_{i}^*-g(Z_{i};\widehat{\beta }^*)\right) \mid \overline{Z}\right]
=F_{\varepsilon^*}\left( g(\,Z_{i};\widehat{%
\beta }^*)-g(Z_{i};\widehat \beta)\mid \overline{Z}\right)
-\tau ,
$$
and for any sample $\overline{Z}$, the distribution function $F_{\varepsilon^*}( \cdot\mid \overline{Z})$ is that of the uniform law on $[-\tau, 1-\tau],$
the boundedness of $\sqrt{n} (\widehat\beta^* -\widehat\beta)$ follows by a Taylor expansion of $F_{\varepsilon^*}( \cdot\mid \overline{Z})$ around the origin, exactly like in the proof of Lemma \ref{conv_alt} in the case $r_n=0$. The case of the wild bootstrap and linear quantile regression follows as a consequence of Theorem 1 of \cite{Feng2011}. The arguments of Theorem 1 of \cite{Feng2011} could be adapted to nonlinear models using  a linearization like in the proof of Lemma \ref{conv_alt}. The details are omitted.

Finally, using Lemma \ref{sher}, derive conditional versions of
Lemma 1 of \cite{Zheng1998} and of Corollary 4 of Sherman in the case of
constant envelopes. Combine these results with the fact that $\sqrt{n}(\widehat \beta^* - \widehat \beta)$ is bounded in probability given the sample  and follow
the lines of the proof of Theorem \ref{as_law} above to deduce that for any $%
\varepsilon >0$%
\begin{equation*}
\mathbb{P}\left( nh^{1/2}\left\vert W_{n}^{\ast }(\widehat{\beta }^{\ast
})-W_{n}^{\ast }(\widehat{\beta })\right\vert >\varepsilon \mid
Y_{1},Z_{1},...,Y_{n},Z_{n}\right) \rightarrow 0,\quad \text{in probability.}
\end{equation*}%

\newpage

\begin{figure}
\includegraphics[scale=0.6,angle=270]{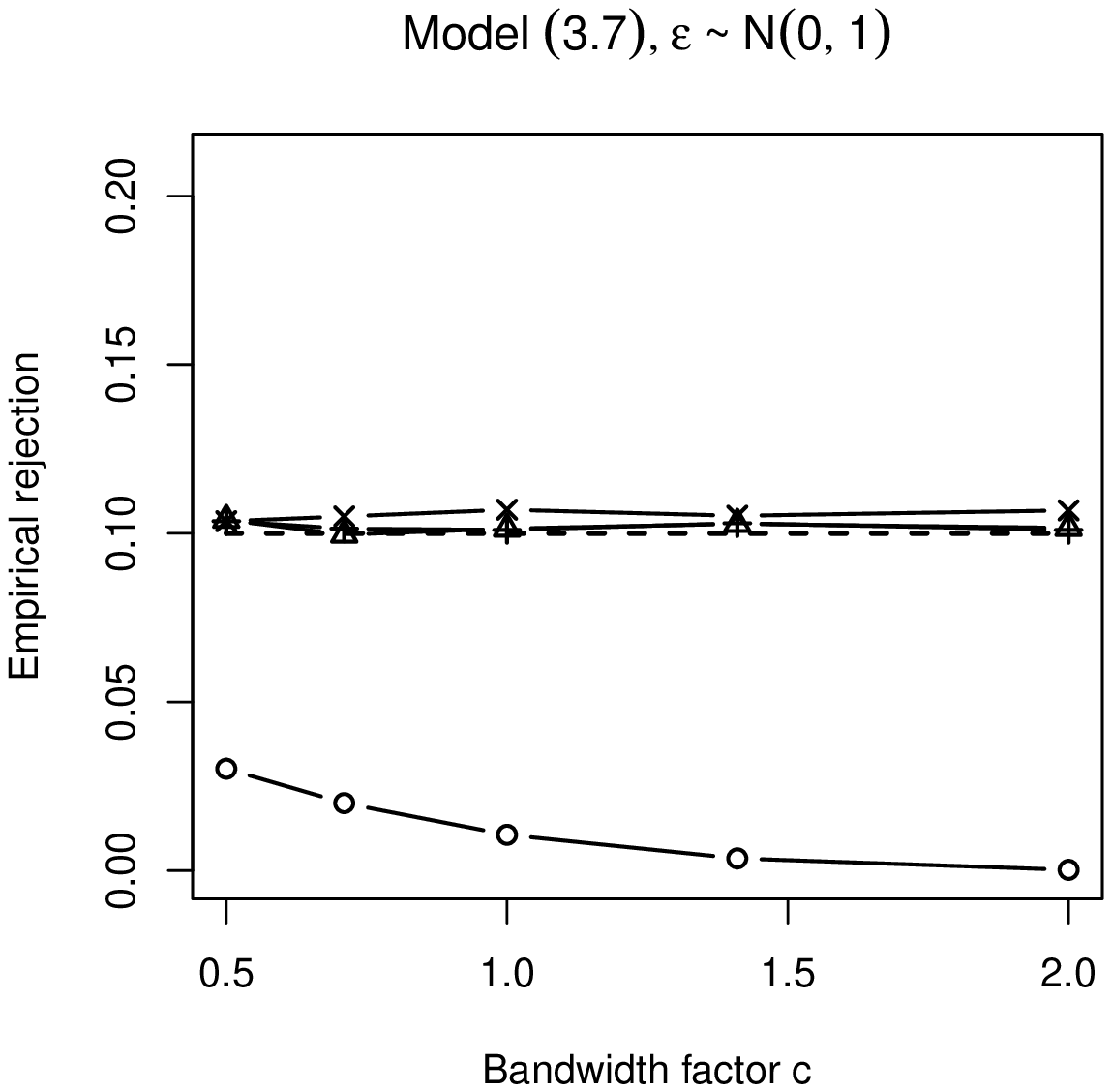}\includegraphics[scale=0.6,angle=270]{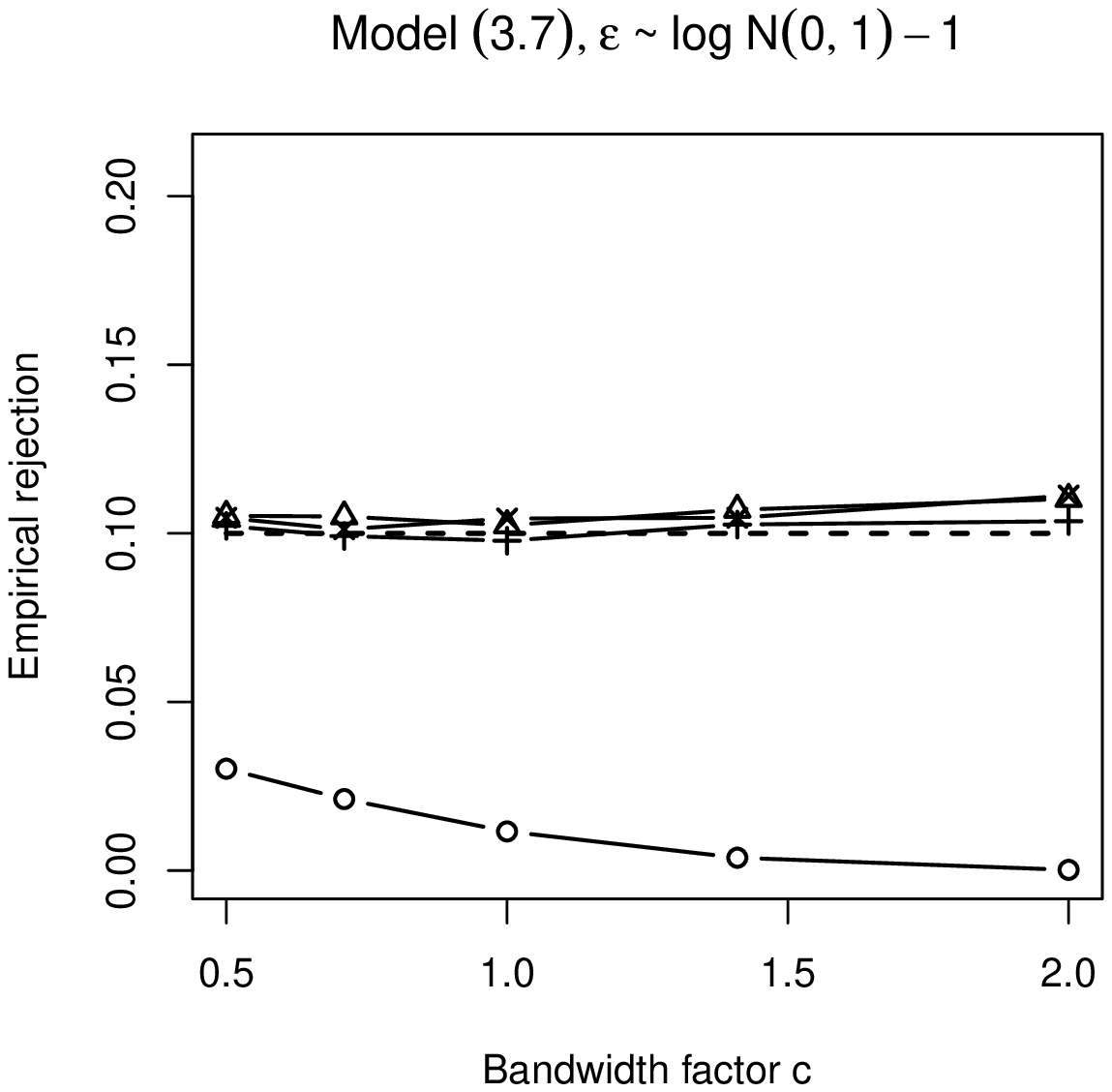}

\includegraphics[scale=0.6,angle=270]{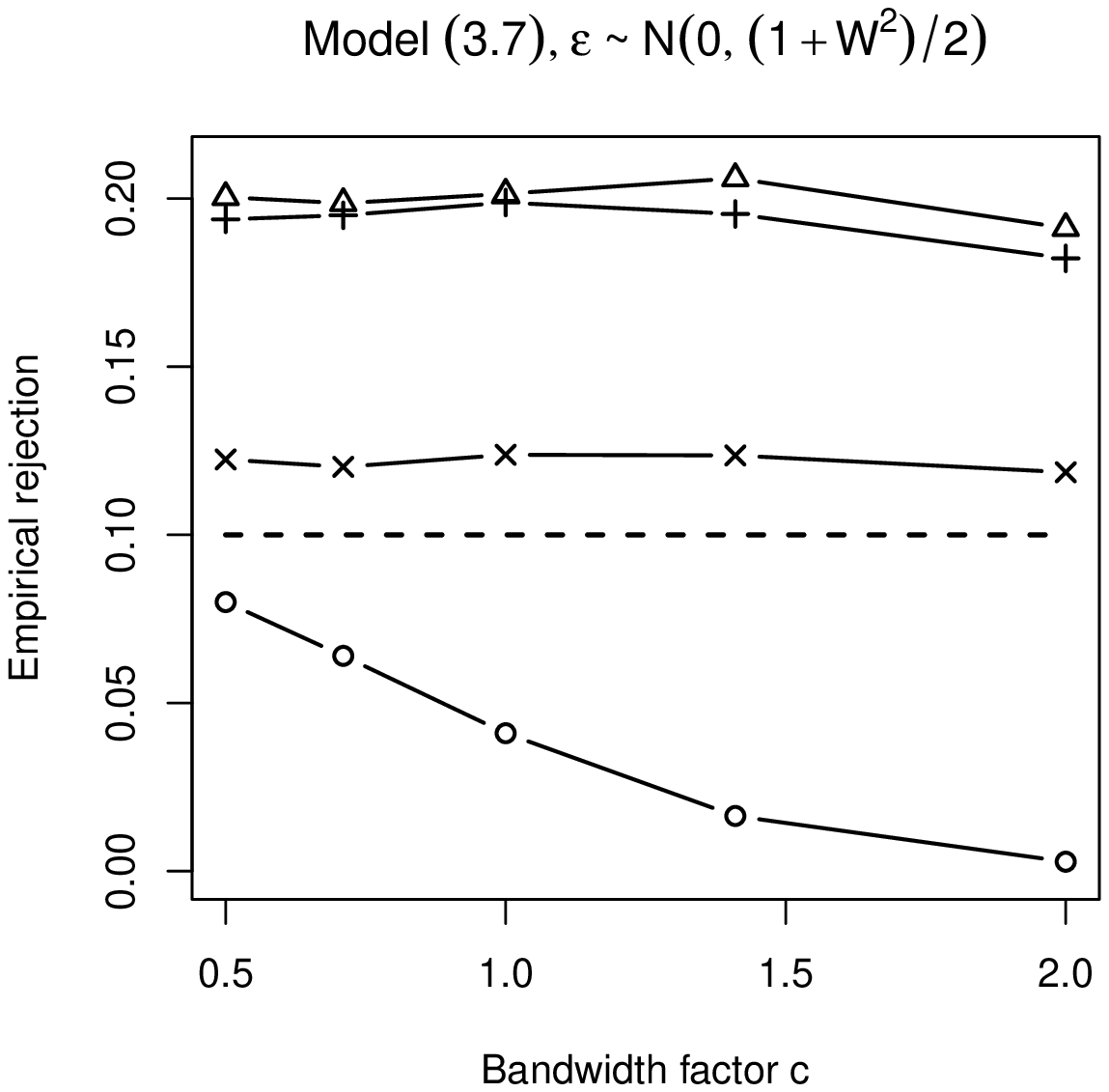}\includegraphics[scale=0.6,angle=270]{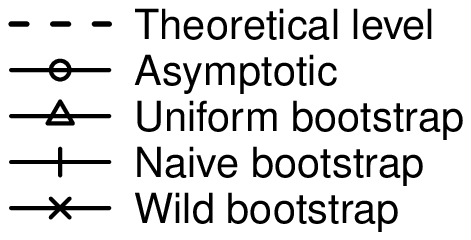}

\caption{Empirical rejections under $H_{0}$ with model (\ref{eq:setup1}) as a function of the bandwidth,
$n=100$\label{fig:Level}}
\end{figure}
\begin{figure}
\includegraphics[scale=0.6,angle=270]{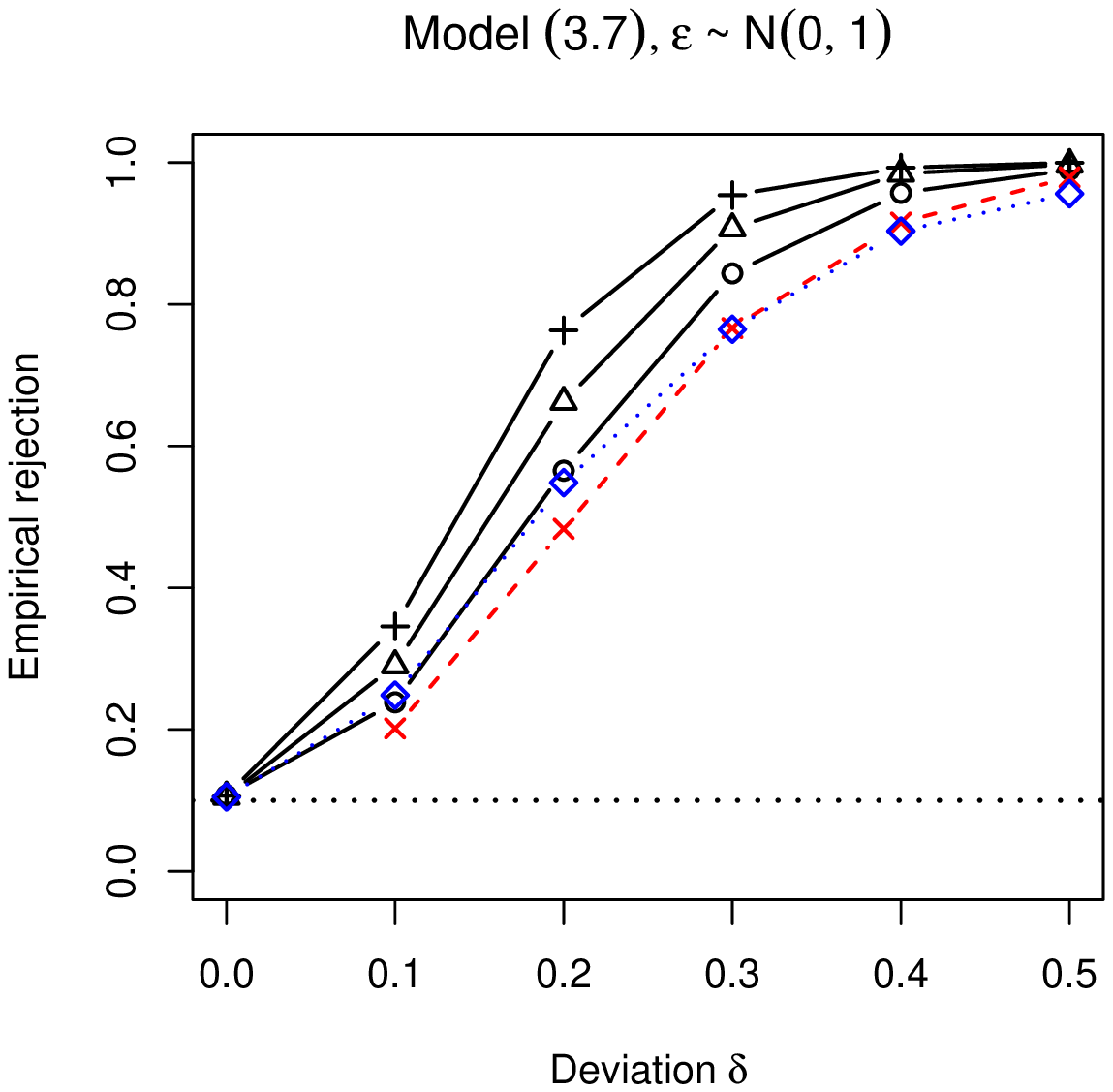}\includegraphics[scale=0.6,angle=270]{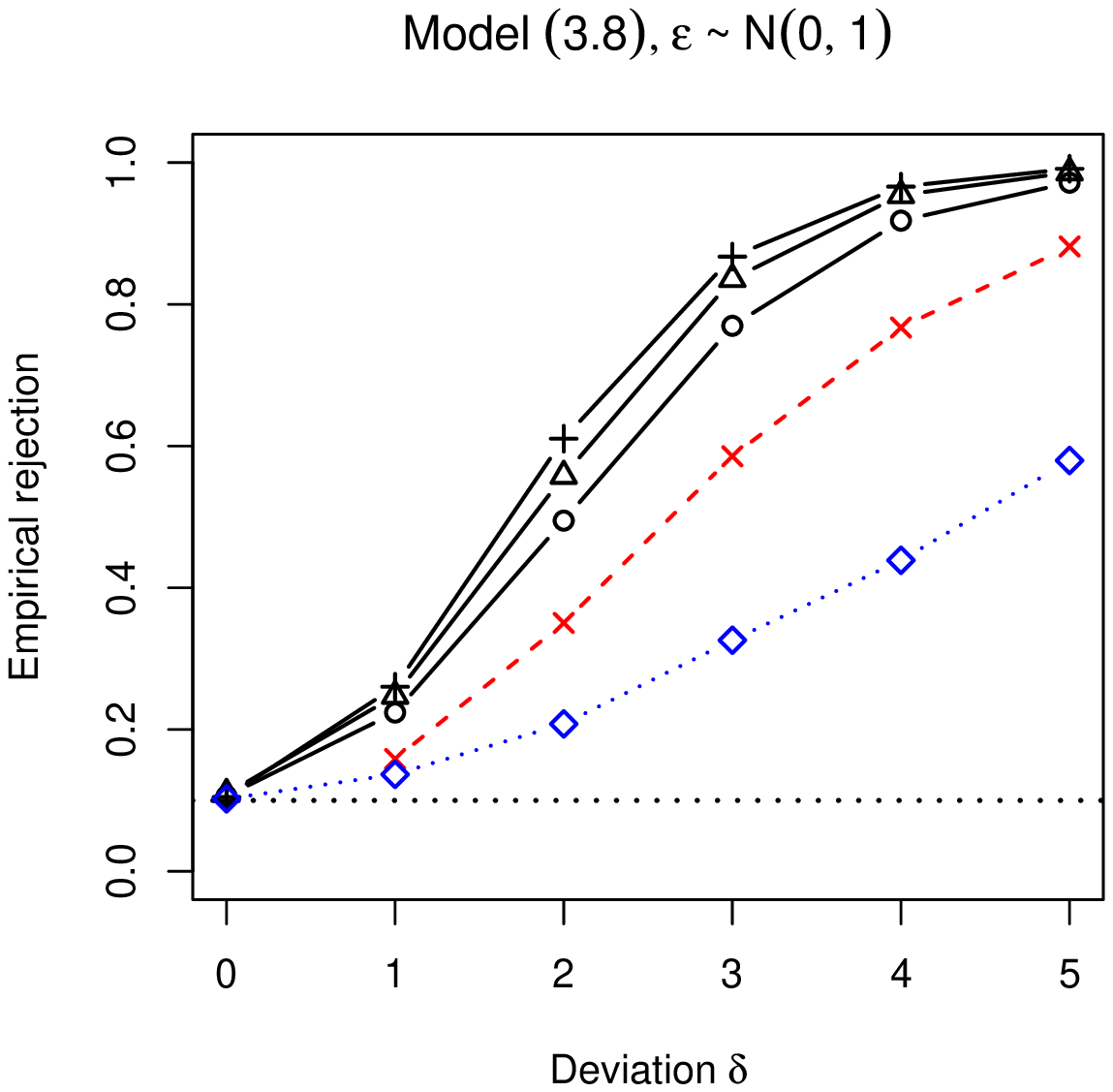}

\includegraphics[scale=0.6,angle=270]{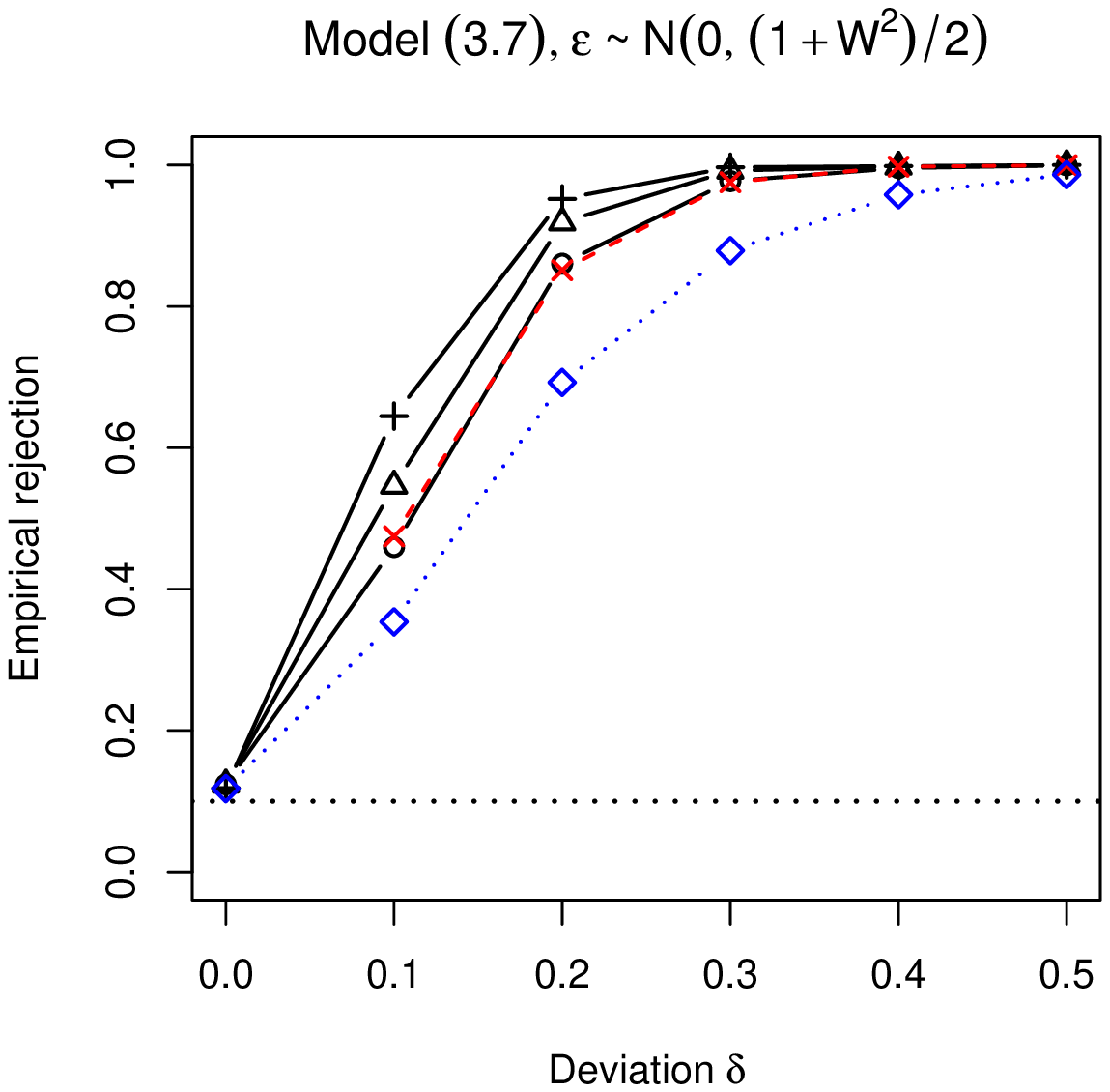}\includegraphics[scale=0.6,angle=270]{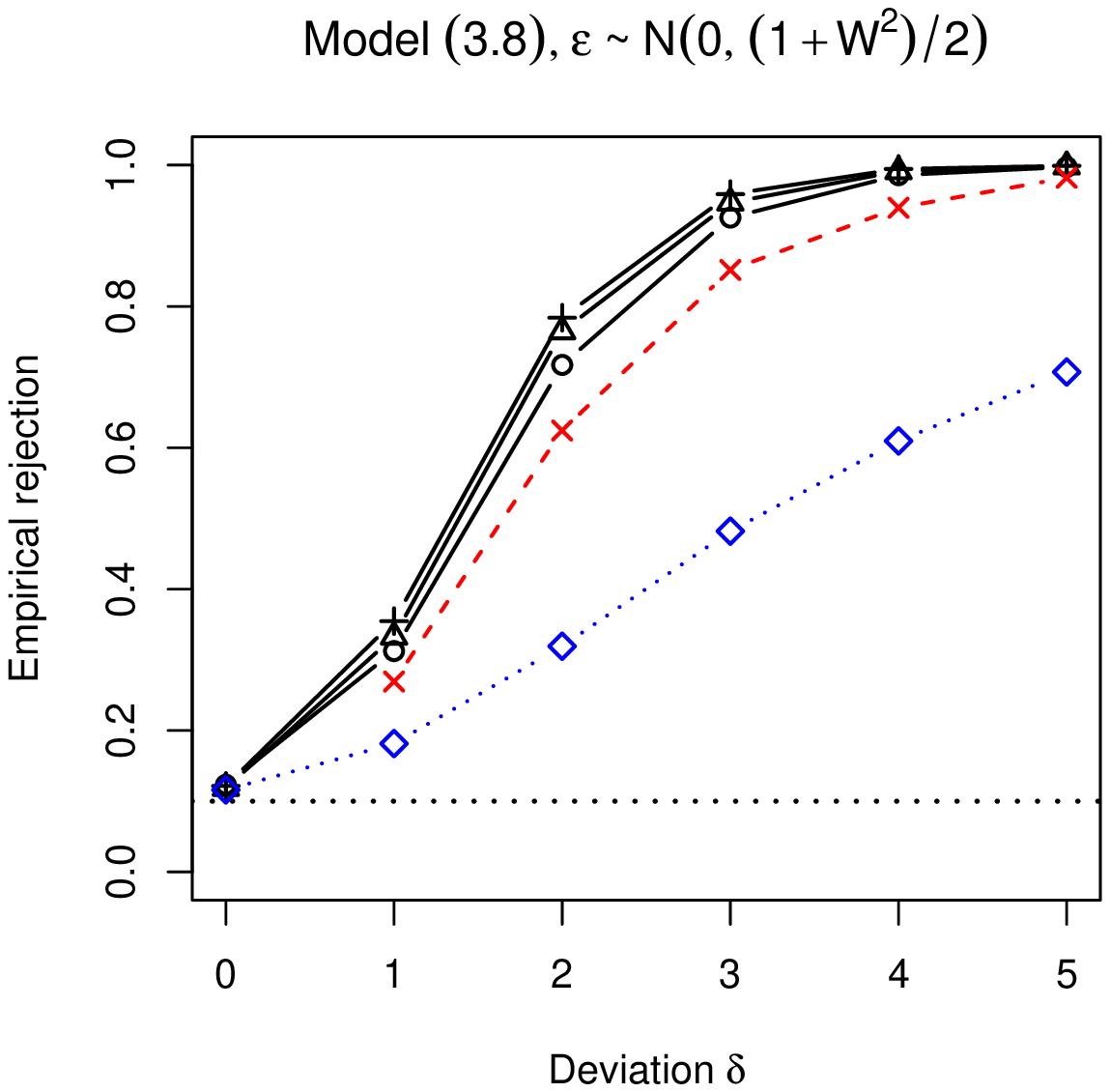}

\begin{centering}\includegraphics[bb=190bp 420bp 420bp 0bp,clip,scale=0.6,angle=270]{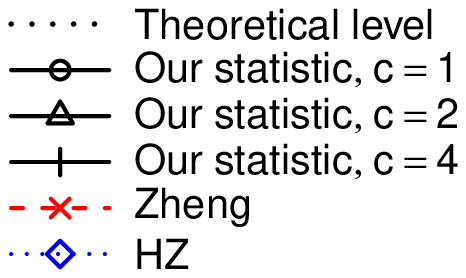}\par\end{centering}

\caption{Power curves for models (\ref{eq:setup1}) and (\ref{eq:setup2}), $n=100$.\label{fig:Power}}
\end{figure}

\newpage

\begin{table}
\caption{Application: estimation results and tests
  p-values \label{tab:Regression}}
\begin{tabular}{rrrrrrr}
  \hline
&  $\tau=0.5$  &  $\tau=0.1$
&  $\tau=0.5$  &  $\tau=0.1$
&  $\tau=0.5$  &  $\tau=0.1$
\\
  \hline
  CIGAR & -5.35 & -7.53 & -5.05 & -8.36 & -5.07 & -8.07 \\
   & (2.28) & (4) & (2.3) & (3.53) & (2.36) & (3.25) \\
  WTGAIN & 8.09 & 14.73 & 7.69 & 14.96 & 8.31 & 15.91 \\
   & (1.33) & (0.75) & (1.32) & (1.2) & (1.31) & (1.4) \\
AGE & -9.34 & -5.13 & 43.6 & 133.67 & 78.59 & 117.62 \\
   & (3.82) & (4.47) & (50.59) & (30.11) & (45.85) & (48.42) \\
  AGESQ &  &  & -0.84 & -2.23 & -1.38 & -1.94 \\
   &  &  & (0.81) & (0.5) & (0.72) & (0.82) \\
  BOY &  &  &  &  & 137.22 & -5.22 \\
   &  &  &  &  & (34.35) & (47.33) \\
  BLACK &  &  &  &  & -177.78 & -124.18 \\
   &  &  &  &  & (75.09) & (69.17) \\
  MARRIED &  &  &  &  & 21.62 & 41.75 \\
   &  &  &  &  & (48.39) & (54.66) \\
  NOVISIT &  &  &  &  & -211.62 & -275.15 \\
   &  &  &  &  & (406.72) & (112.5) \\
\hline
HZ  & 0.347 & 0.227 & 0.266 & 0.356 & 0.272 & 0.135 \\
Our test c=1 & 0.791 & 0.165 & 0.738 & 0.942 & 0.068 & 0.972 \\
Our test c=2 & 0.704 & 0.044 & 0.741 & 0.968 & 0.078 & 0.796 \\
 \hline
\end{tabular}
\end{table}

\end{document}